\documentclass[11pt,reqno]{amsart}
\usepackage{various, xcolor, hyperref, tikz, pgfplots}

\mathtoolsset{showonlyrefs}
\usepackage[centering, margin=1.02in]{geometry}

\pgfplotsset{compat=1.18}

\newcommand*{\bbone}{\text{\usefont{U}{bbold}{m}{n}1}}

\begin{document}

\author{Alexandre de Faveri}
\address{EPFL SB MATH, Station 10, 1015 Lausanne, Switzerland}
\email{\url{alexandre.defaveri@epfl.ch}}

\author{Mayank Pandey}
\address{Courant Institute of Mathematical Sciences, 251 Mercer Street, New York, NY, 10012}
\email{\url{mayankpandey9973@gmail.com}}

\title{The sixth moment of the Riemann zeta function}

\begin{abstract}
  We prove new large value estimates for the Riemann zeta function on the critical line. For instance, we improve the upper bound on the measure of $t\in [T, 2T]$ with $|\zeta(1/2 + it)| \geq T^{1/8}$ for the first time since Hardy--Littlewood (1923). Our results imply the first improvement since work of Heath-Brown (1978) on upper bounds for the $k$-th moment of zeta, for every $4 < k \leq 12$. We show in particular that
  \begin{equation*}
    \int_{0}^{T} |\zeta(\tfrac{1}{2} + it)|^6 \,d t \ll_\varepsilon T^{\frac{5}{4} - \frac{1}{60} + \varepsilon}.
  \end{equation*}
  The main ingredient is a new large value estimate for exponential sums with square-root phases, and certain perturbations. Those arise in the expansion of the short second moment of zeta.
\end{abstract}

\maketitle


\section{Introduction}

The Lindel{\"o}f hypothesis predicts that $|\zeta(1/2 + it)|\ll_\varepsilon (1  + |t|)^{\varepsilon} $ for all $\varepsilon > 0$. It is a consequence of the Riemann hypothesis, and can act as a substitute for it in many applications. A measure of progress towards Lindel{\"o}f is provided by upper bounds for the moments 
\begin{equation*}
    M_k(T) := \int_0^T |\zeta(\tfrac{1}{2}+it)|^k\, dt.
\end{equation*}
For instance, the bound $M_k(T) \ll_{k, \varepsilon} T^{1+\varepsilon}$ for all $k \geq 0$ is equivalent to the Lindel{\"o}f hypothesis. Such a bound is only known for $0 \leq k \leq 4$, but partial progress towards it is useful in applications.

The classical fourth and twelfth moment bounds\footnote{In both bounds, $T^\varepsilon$ can be replaced with a power of $\log{T}$. We are concerned only with the exponents of $T$ in this work, so we simplify the bounds for convenience.} 
\begin{equation}\label{eq:classical_moments}
  M_4(T)  \ll_\varepsilon T^{1+\varepsilon} \qquad\text{and} \qquad M_{12}(T)  \ll_\varepsilon T^{2+\varepsilon},
\end{equation}
due respectively to Hardy--Littlewood \cite{MR1575368} (refined by Ingham \cite{MR1575391}) and Heath-Brown \cite{MR517737}, imply by interpolation that $M_k(T) \ll_{\varepsilon} T^{1 + (k-4)/8+\varepsilon}$ for $4 \leq k \leq 12$. Those have been the best known bounds since the work of Heath-Brown \cite{MR517737}, giving in particular $M_6(T) \ll_\varepsilon T^{5/4 + \varepsilon}$.

In this paper, we show the following improved estimate for the sixth moment of zeta.
\begin{theorem}[Sixth moment bound]\label{cor:sixth_moment}
  For every $T \geq 1$ and $\varepsilon > 0$, we have
  \begin{equation*}
    \int_{0}^T |\zeta(\tfrac{1}{2} + it)|^6 \,d t \ll_\varepsilon T^{\frac{5}{4} - \frac{1}{60} + \varepsilon}.
  \end{equation*}
\end{theorem}

We also improve the best known upper bound for the $k$-th moment of zeta, for every $4 < k \leq 12$.

\begin{theorem}[General moment bound]\label{cor:gen_moments}
  For each $4 < k \leq 12$, there exists $\delta_k > 0$ such that if $T \geq 1$ and $\varepsilon > 0$, then
  \begin{equation*}
    M_{k}(T) \ll_\varepsilon T^{1 + \frac{k-4}{8} - \delta_k + \varepsilon}.
  \end{equation*}
\end{theorem}

In Theorem \ref{theorem:curykapvmu}, we show that $\delta_{12} = \frac{520}{445121} $ is admissible. Improved moment bounds can also be obtained for $12 \leq k < 12 + \nu$ for a small $\nu > 0$. We do not optimize either the savings $\delta_k$ or the range $\nu$ in this paper, aiming instead to simplify the presentation.

Our results on moments derive from distributional estimates for $|\zeta(\tfrac{1}{2}+it)|$. Let $N(T, V) $ be the size of the largest $1$-spaced subset of $\{t\in [T, 2T] : |\zeta(\tfrac{1}{2}+it)| \geq V \}$. The classical bounds \eqref{eq:classical_moments} imply
\begin{equation}\label{eq:fourth_twelfth_moment_bound}
  N(T, V)\ll_\varepsilon T^{\varepsilon}\min(TV^{-4}, T^2 V^{-12}).
\end{equation}
The Hardy--Littlewood bound $N(T, V)\ll_\varepsilon T^{1+\varepsilon}V^{-4}$ is the best known (up to factors of $T^\varepsilon$) for $V \leq T^{1/8}$, and the bounds in \eqref{eq:fourth_twelfth_moment_bound} intersect at $V = T^{1/8}$. Better results are available for $V > T^{1/8}$ using exponent pairs \cite{MR4864482}, so $V = T^{1/8}$ is the critical range for known bounds on the moments $M_k(T)$ when $4 < k \leq 12$. Our main result gives an improvement in and around this critical range. 
\begin{theorem}[Large value estimate]\label{theorem:curvbfoast}
  Let $T\geq 1$. For any $V\in [T^{7/60}, T^{5/38}]$ and $\varepsilon > 0$, we have 
  \begin{equation*}
    N(T, V) \ll_\varepsilon T^{\frac{24}{17} + \varepsilon}V^{-\frac{128}{17}}. 
  \end{equation*}
\end{theorem}

In \S\ref{sec:curvbfnzzp} we record further bounds for $N(T, V) $, which are superior to Theorem \ref{theorem:curvbfoast} for large $V$.

\subsection{On the sixth and twelfth moments}

The estimate $M_{12}(T) \ll_\varepsilon T^{2 + \varepsilon}$ was proved by Heath-Brown via short second moments of zeta. The desire to improve on this bound has inspired many developments in analytic number theory, and led to several different proofs of the same result. For example, Iwaniec \cite{zbMATH03683696} deduced it from bounds for the fourth moment of zeta in short intervals, obtained via spectral theory of automorphic forms. Jutila \cite{Jutila1991MeanValues} gave a different proof of Iwaniec's result, bypassing the appeal to spectral theory. A textbook proof is given in \cite[Chapter 8]{MR1994094}.

Both Iwaniec's and Jutila's methods give the same bounds for a fixed $\mathrm{GL}_2 $ $L$-function in place of $\zeta^2 $.
It is worth mentioning that significantly less is known for $\mathrm{GL}_3 $ $L$-functions. The best
known bound for the second moment in the $t$-aspect (which should be compared to $M_6(T)$) is $T^{4/3 + \varepsilon} $, due to Dasgupta, Leung, and Young
\cite{DasguptaLeungYoung2024}. Interestingly, the bound $T^{5/4 + \varepsilon} $ was recently obtained by Aggarwal, Leung, and Munshi \cite{AggarwalLeungMunshi2026} for a short second moment of length $T^{2/3} $. 

As noted earlier, the sixth moment bound of $O_\varepsilon(T^{5/4 + \varepsilon}) $ follows by interpolating the fourth and twelfth moment
bounds.  Expanding out with the approximate functional equation, this bound roughly amounts to showing an asymptotic with remainder $O_\varepsilon(X^{7/6 + \varepsilon}) $ for a smoothed version of
\begin{equation}
  \sum_{h\sim X^{1/3} } \sum_{n\sim X } d_3(n)d_3(n + h).
\end{equation}
One may interpret this $7/6 $ as square-root cancellation in the $h $-sum, over the $n$-sums minus their expected main terms, and Theorem \ref{cor:sixth_moment} should improve the exponent $7/6$.

Finally, let us mention that sufficiently large improvements on the sixth moment bound
imply improved zero density estimates for zeta; Iwaniec \cite{Iwaniec2024PartialSixthMoment} noted that a bound $M_6(T) \ll T^{6/5 - \delta} $ implies better zero density estimates at $\sigma = 3/4 $ than Huxley's \cite{MR292774} classical bound. Such an improvement was recently obtained via different methods by
Guth and Maynard \cite{MR5037841}. 

\subsection{Large value bounds for special exponential sums}

The key ingredient behind our bounds is a new large value estimate for exponential sums with
certain special phases. We refer to the survey \cite{Guth2025LargeValues} for the importance of large value problems in a broad range of subjects and discussion of their difficulty. In this paper we show the following. 
\begin{corollary}\label{theorem:curtz38o04}
  Let $N, W, V \geq 1$, $\eta \in \mathbb{R}$, and $a_n\in \mathbb{C}$ with $|a_n|\leq 1$. Denote $\Upsilon = 1 + |\eta|W$ and
  \begin{equation*}
    D_\eta(w) := \sum_{N < n\le 2N } a_n e(w\phi_\eta(n/N)) \qquad \text{for} \qquad \phi_\eta(x) := \sqrt{x} + \eta x^{3/2}.
  \end{equation*}
  Suppose that $N\le W\le N^{\frac{4}{3}}$ and that $|\eta| \le c W/N^2 $ for some sufficiently small $c > 0 $. 
  Then for $\varepsilon > 0 $ and any $1$-spaced set $\mathcal{W} \subset [0, W]$ such that $|D_\eta(w)|\ge V $ for all $w\in \mathcal{W} $, we have
  \begin{equation}\label{eq:curud6bizo}
    |\mathcal{W}| \ll_\varepsilon N^\varepsilon\biggl( \frac{N^2}{V^2} + \frac{N^2W^{\frac{3}{4}}\Upsilon^{\frac{1}{8}}}{V^3} + \frac{N^{\frac{5}{2}}W^{\frac{3}{2}}\Upsilon^{\frac{1}{4}}}{V^5}\biggr).
  \end{equation}
\end{corollary}

For reference, in the more general situation where $\phi_\eta $ is replaced with a smooth $\phi$ satisfying
$|\phi'| \asymp |\phi''|\asymp 1 $ in a fixed neighborhood of $[1, 2] $, one can show using the large sieve \cite[Corollary 2]{MR337775}
that \eqref{eq:curud6bizo} may be replaced with
\begin{equation}\label{eq:curud7bumr}
  |\mathcal{W}| \ll \frac{N^2}{V^2} + \frac{NW}{V^2}.
\end{equation}
Heath-Brown's arguments in \cite{MR517737} (of the type used by Hal\'asz \cite{MR230694} and Montgomery \cite{MR268130} for Dirichlet
polynomials) yield that
\begin{equation}\label{eq:curud7bold}
  |\mathcal{W}| \ll \frac{N^2}{V^2} + \frac{N^4W}{V^6}.
\end{equation}
For $N\gg W $, the term $N^2/V^2 $ is optimal, and for $N\ll W $, the bounds \eqref{eq:curud7bold}
and \eqref{eq:curud7bumr} meet at $V\asymp N^{3/4} $, a regime often critical for large value estimates in analytic
number theory. It was so in Guth--Maynard's \cite{MR5037841} recent work on large values of
Dirichlet polynomials, and is for us too. 

For the sake of clarifying Corollary \ref{theorem:curtz38o04} in the special case $V \asymp N^{\frac{3}{4}} $ and $N^{1+\delta}\ll W\ll N^{\frac{5}{4}} $, we shall isolate its conclusion when $|\eta|\ll W/N^{2+\delta} $, for a small $\delta > 0$. In this case, it gives the bound 
\begin{equation}\label{eq:1}
  |\mathcal{W}|\ll_\varepsilon N^\varepsilon \frac{W}{\sqrt{N}}\biggl(\frac{W}{N}\biggr)^{-\frac{1}{4}} (1 + |\eta| W) ^{\frac{1}{8}} \ll \frac{W}{\sqrt{N}} N^{-\frac{\delta}{8} + \varepsilon}, 
\end{equation}
while the mean value bound \eqref{eq:curud7bumr} yields $|\mathcal{W}|\ll W/\sqrt{N} $. Thus our result is nontrivial in this case. Since the relevant regime for the zeta function is $|\eta|\asymp 1/N  $, Corollary \ref{theorem:curtz38o04} suffices for our purposes. 

For applications to large values of zeta, the main advantage of considering exponential sums with our special phases is that it then suffices to beat their mean value estimate. This does not seem to be the case if one directly converts the problem to large values of Dirichlet polynomials (e.g.\ via the approximate functional equation), so available bounds fall short of what is needed.

That we have an $\ell^\infty $ bound on $(a_n)  $ rather than an $\ell^2 $ bound is essential, otherwise
the estimate $|\mathcal{W}|\ll W/\sqrt{N} $ is sharp when $\eta = 0$: we may take $\mathcal{W} = \{j\sqrt{N} : 1 \le j \le W/\sqrt{N}\} $ and $a_n = N^{1/4} \cdot \bbone_{ \square }(n)$, for instance. This counterexample was noted by Guth \cite[\S5.5]{Guth2025LargeValues}
and comes from a simplification by Alex Cohen of a related counterexample due to Fu--Guth--Maldague \cite{FuGuthMaldague2023}.

\subsection{Notation}

Let $e(x) := e^{2\pi i x}$. Our Fourier transforms are $\widehat{f}(\xi) := \int_\mathbb{R} f(x) e(-\xi x)\, dx$ by convention. For $y\in \mathbb{R}$, denote $\|y\| = \min \{|y-n| : n \in \mathbb{Z}\}$. We write $f \ll g$ if there is a constant $C > 0$ such that $|f| \leq C |g|$. If $C$ depends on a parameter $p$, we write $f \ll_p g$. The notation $f \asymp g$ means $f \ll g \ll f$, and $a \sim A$ denotes $A < a \le 2A$. Finally, for a finite set $\mathcal{W} \subset \mathbb{R}$, we let
\begin{align*}
  E_{2k}(\mathcal{W}; D) := \#\biggl\{(w_1,\dots,w_{2k})\in \mathcal{W}^{2k} : \biggl|\sum_{i=1}^k w_i - \sum_{i = k +1}^{2k} w_i \biggr|\le D  \biggr\}.
\end{align*}


\subsection*{Acknowledgments} 

The authors thank Maksym Radziwi{\l}{\l} for many interesting discussions on the topics of this paper, and for his advice and encouragement. We are also grateful to Valentin Blomer and Larry Guth for helpful comments on a previous version of the paper.


\subsection*{Use of AI tools}

The reduction of the problem to beating the mean-value bound for exponential sums with perturbed square-root phases, and the observation that the high-energy case can be controlled as in Guth--Maynard \cite{MR5037841}, are due to the authors. Using the trace method (with third powers) as in \cite{MR5037841}, and dealing with stationary points via an extra pruning step, we initially obtained a result of the same strength as Corollary \ref{theorem:curtz38o04} in the case $\eta = 0$. This result did not rely on AI tools, using them only in a subsidiary capacity (e.g.\ to quickly compute stationary points and optimize exponents), but was only nontrivial for $|\eta| = o(1/N) $, just barely insufficient for our desired application to large values of $|\zeta(1/2 + it)|$.
We then asked ChatGPT 6 Astra to extend the range of $\eta$ in our proof, and it autonomously produced a simpler argument which is nontrivial for $|\eta| = o(W/N^2) $. A refinement of that argument is used in this paper.


\section{Sketch of the proof}

In this section we give a sketch of the argument in the critical case $V = T^{1/8} $. For readability, we shall be somewhat imprecise, for instance ignoring factors of $T^\varepsilon $ and constant factors in phases.

By work of Heath-Brown \cite{MR517737}, it suffices to rule out the following situation: there exist $T^{1/2} $ points $t \in [T, 2T] $ which are $T^{1/4} $-spaced\footnote{We can even upgrade this spacing to $T^{1/2}$ via \cite{zbMATH03683696}, but ignore this in the sketch for simplicity.} and such that $|\zeta(1/2 + it)|\ge T^{1/8}$. 

For $N \asymp T/V^4 \asymp T^{1/2}$ and $W_0 \asymp \sqrt{TN} \asymp T^{3/4} $, Atkinson's formula would then roughly imply the existence of a set $\mathcal{W}_0\subset [W_0, 2W_0] $ of $1 $-spaced points such that $|\mathcal{W}_0| \asymp W_0/\sqrt{N} \asymp T^{1/2}$ and 
\begin{equation*}
  \biggl|\sum_{n\sim N } d(n)e \biggl( w\sqrt{\frac{n}{N}} + \frac{T}{w} \biggl(\frac{n}{N}\biggr)^{3/2} \biggr) \biggr| \gg N^{3/4} \qquad \text{for all } w\in \mathcal{W}_0.
\end{equation*}

Recall that $\phi_\eta(x) = \sqrt{x} + \eta x^{3/2} $. We subdivide $\mathcal{W}_0 $ into intervals of length $W = N^{1 + \delta} $, for some sufficiently small $\delta >0$. It then suffices\footnote{The subdivision into shorter intervals is needed to make $\eta$ independent of $w$, and to remove lower order terms.} to show, for any coefficients $|a_n| \leq 1$ and $\eta\in \mathbb{R}$ with $|\eta|\asymp T/W_0^2 \asymp 1/N$, that there exists no $1 $-spaced subset
$\mathcal{W}\subset [W, 2W] $ with $|\mathcal{W}| \gg W/\sqrt{N} $ and
\begin{equation}\label{eq:large_val_sketch}
  \biggl| \sum_{n\sim N } a_n e(w\phi_\eta(n/N)) \biggr|\gg N^{3/4} \qquad \text{for all } w\in \mathcal{W}.
\end{equation}
Thus we are precisely in the position of having to beat the mean value bound on $R := |\mathcal{W}|$. 

Our proof begins similarly to that of the Hal\'asz--Montgomery bound for Dirichlet polynomials (like in Heath-Brown's work on the twelfth moment), with an application of duality. Precisely, for some phases $\gamma_w$ with $|\gamma_w| = 1$, we have by \eqref{eq:large_val_sketch} and Cauchy--Schwarz that
\begin{equation*}
  R^2 N^{3/2} \ll \biggl| \sum_{\substack{n\sim N \\  w\in \mathcal{W}} } \gamma_w a_n e(w\phi_\eta(n/N)) \biggr|^2
  \ll N \sum_{n\sim N}  \biggl| \sum_{w\in \mathcal{W} } \gamma_w e(w \phi_\eta(n/N)) \biggr|^2.
\end{equation*}
Expanding out the square, separating out the diagonal (assuming for simplicity that the off-diagonal consists only of the generic case $|w - w'|\asymp W $), and applying Poisson summation in $n$ followed by stationary phase, we have that the right-hand side above is
\begin{equation*}
  \ll N^2 R + \frac{N^2}{\sqrt{W}}
  \sum_{|m| \sim W/N }\sum_{\substack{w, w'\in \mathcal{W}}}
  \gamma_w \overline{\gamma_{w'}}
  e \biggl( \frac{(w - w')^2}{4Nm} + \eta\frac{(w - w')^4}{8N^3m^3} \biggr).
\end{equation*}
At this point we apply H\"older's inequality, so that the above is
\begin{equation*}
  \ll N^2 R + \frac{N^2}{\sqrt{W}} \sum_{|m|\sim W/N}
  R^{3/4} \biggl( \sum_{w'\in \mathcal{W} }
  \biggl| \sum_{w\in \mathcal{W} } \gamma_w
  e \biggl( \frac{(w - w')^2}{4Nm} + \eta\frac{(w - w')^4}{8N^3m^3} \biggr) \biggr|^4\biggr)^{1/4}. 
\end{equation*}
We may bound (with a $T^{\varepsilon} $ loss) the sum over $w' \in \mathcal{W}$ by an integral over $w' \sim W$. Denoting $w' = Wx $ and $\Delta_j := w_1^{j} + w_2^j - w_{3}^j - w_{4}^j$, this gives
\begin{align}
   &\sum_{w'\in \mathcal{W} }
    \biggl| \sum_{w\in \mathcal{W} } \gamma_w
    e \biggl( \frac{(w - w')^2}{4Nm} + \eta\frac{(w - w')^4}{8N^3m^3} \biggr) \biggr|^4 
  \ll \ W \int_{x\sim 1 }
      \biggl| \sum_{w\in \mathcal{W} } \gamma_w
    e \biggl( \frac{(w - Wx)^2}{4Nm} + \eta\frac{(w - Wx)^4}{8N^3m^3} \biggr) \biggr|^4 \,d x \nonumber
  \\
  & \ll 
  W\sum_{ w_1, w_2, w_3, w_4 \in \mathcal{W} } \ \biggl|
  \int_{x\sim 1}
  e \biggl( -\frac{W \Delta_1}{2Nm} x + \eta W \biggl[ -\frac{\Delta_1 W^2}{2N^3m^3}x^3 +  \frac{3\Delta_2 W}{4 N^3m^3}x^2 - \frac{\Delta_3}{2 N^3m^3}x \biggr] \biggr) \,d x\biggr|. \label{multline:curylcfx6w}
\end{align}

If $|\Delta_1|$ is greater than a sufficiently large multiple of $ |\eta| W$, then the first term in the phase
dominates, hence the contribution of such $(w_1,\dots,w_4) $ is negligible. This makes critical use of the shape of our initial phases $\phi_\eta(x)$. If $|\Delta_1|\ll |\eta| W$, the integral is $\ll (1 + |\Delta_1|)^{-1/2} $ (modulo some minor singular cases) by stationary phase.
Thus \eqref{multline:curylcfx6w} is $\ll W E_4(\mathcal{W}; 1) (1  +|\eta|W)^{1/2}$, so
\begin{equation*}
  R^2N^{3/2} \ll N^2R + \frac{N^2}{\sqrt{W}} \frac{W}{N} R^{3/4} W^{1/4} E_4(\mathcal{W}; 1)^{1/4} (1 + |\eta| W)^{1/8}.
\end{equation*}

With an argument similar to that of Guth and Maynard \cite[\S11]{MR5037841}, we can show that
$E_4(\mathcal{W};1)\ll NR $. This is the only place in the argument where we make (crucial) use of the $\ell^\infty$ bound $|a_n|\leq 1$, instead of a weaker $\ell^2$ bound. We conclude that
\begin{equation*}
  R\ll N^{1/2} + \frac{W}{\sqrt{N}} \biggl(\frac{W}{N}\biggr)^{-\frac{1}{8}} \ll \frac{W}{\sqrt{N}} \biggl(\frac{W}{N}\biggr)^{-\frac{1}{8}},
\end{equation*}
which gives the desired improvement over the mean value bound $R \ll W/\sqrt{N}$.


\section{Preliminary lemmas}

\begin{lemma}[Stationary phase]\label{lemma:curyjy52ai}
  For any bounded interval $I\subset \mathbb{R} $ and $\psi\in C_c^\infty(\mathbb{R}) $ with support contained in $I $, we have that 
  \begin{equation*}
    \int_\mathbb{R} \psi(x) e(\alpha x + \beta x^2) \,d x \ll_I (\| \psi \|_\infty + \| \psi' \|_1) (1 + |\alpha| + |\beta| )^{-1/2}
  \end{equation*}
  uniformly in $\alpha, \beta\in \mathbb{R} $.
\end{lemma}

\begin{proof}
  This follows immediately from stationary phase, for example as a special case of \cite[Chapter~VIII, Proposition~2]{Stein1993}.
\end{proof}

\begin{lemma}[Gallagher's lemma]\label{lemma:curw6ekln5}
  For any $f \in C^1(\mathbb{R})$, $x\in \mathbb{R}$, and integer $k\geq 1$, we have 
  \begin{equation*}
    |f(x)|^{2k}\ll_k \int_{x - 1/3}^{x + 1/3} \big(|f(y)|^{2k} + |f(y)|^{2k - 1}|f'(y)|\big) \,d y.
  \end{equation*}
  In particular, if $\mathcal{X}$ is a $1 $-spaced subset of an interval $[a, b] \subset \mathbb{R} $, then 
  \begin{equation*}
    \sum_{ x\in \mathcal{X} } |f(x) |^{2k}\ll_k \int_{a - 1/3}^{b + 1/3}  \big(|f(y)|^{2k} + |f'(y)|^{2k}\big) \,d y.
  \end{equation*}
  
\end{lemma}

\begin{proof}
  This follows from the Sobolev--Gallagher lemma \cite[Lemmas~19.1 and 19.2]{MontgomeryVaughanII}.
\end{proof}


\section{Energy bounds}

Let $N, V, W \geq 1$ and $\eta \in \mathbb{R}$ be bounded. Let $\mathcal{W} \subset [W, 2W]$ be a $1$-spaced set such that $|D_\eta(w)|\ge V$ for every $w \in \mathcal W$. Denote $\xi_n := \phi_\eta(n/N)$ and
\begin{equation*}
  F(\alpha) := \sum_{w \in \mathcal{W}} e(w\alpha).
\end{equation*}

In the next three sections, $f \prec g$ denotes that $|f| \ll_\varepsilon (NVW)^\varepsilon |g|$ for any $\varepsilon > 0$. Subscripts will indicate parameters on which the implied
constants may depend, as usual.

The main result of this section is the following bound.
\begin{proposition}[General energy bound]\label{lem:energy_bound}
  Suppose $W\leq N^{4/3}$ and $\eta \in \mathbb{R}$ is sufficiently small. For every integer $k\geq2$ and real $1\leq D\leq W$, we have
  \begin{equation*}
    E_{2k}(\mathcal W;D) \prec_k \Big(\frac NV \Big)^{4k-4} E_2(\mathcal W; D) +\frac{ND}{V^2}|\mathcal{W}|^{2k-1}.
  \end{equation*}
\end{proposition}

We shall require a few preparatory lemmas before we are ready to prove Proposition \ref{lem:energy_bound}.
\begin{lemma}[Spacing estimate]\label{lemma:frequency_count}
  For $\eta \in \mathbb{R}$ sufficiently small, $0 < \delta < 1 $, and $\lambda\in \mathbb{R}$, we have
  \begin{equation*}
    \# \big\{m, n\sim N : m \neq n, \,\big|\phi_\eta(m/N) - \phi_\eta(n/N) - \lambda\big|\le \delta\big\} \ll \delta N^2 + N^{2/3}. 
  \end{equation*}
\end{lemma}

\begin{proof}
  For $x\asymp 1 $, we record that $\phi'_\eta(x) \asymp \phi_\eta''(x)\asymp 1 $ since $\eta $ is sufficiently small. A direct computation also gives $(\phi_\eta''/(\phi_\eta')^2)'(x) \asymp 1 $.
  
  Denote $\mathcal{S} = \{(m, n)\sim N :  m \neq n,\, |\phi_\eta(n/N) -  \phi_\eta(m/N) - \lambda|\le \delta\} $ and
  $\mathcal{S}_m = \{(m, n)\in \mathcal{S}\} $. We have that $|\mathcal{S}_m| \ll 1 + \delta N $, so  $|\mathcal{S}|\ll N + \delta N^2 $, which is
  sufficient for $\delta \gg 1/N $. 
  
  Suppose that $\delta\le c/N $ for the remainder of this proof, where $c>0$ is sufficiently small. By symmetry, we may assume that $\lambda \ge 0 $. We may also suppose that $1/N \ll \lambda \ll 1$, for otherwise we are forced into the case $n = m $. It remains to bound the number of $m $ for which $\mathcal{S}_m $ is nonempty, since the number of $n$ corresponding to each $m$ is $\ll 1 + \delta N \ll 1$.

  Let $g_\lambda(x) = \phi_\eta^{-1}(\phi_\eta(x) + \lambda) $. It can be checked by implicit differentiation and the mean value theorem that for $x\sim 1 $ we have
  \begin{equation*}
    g_\lambda(x) - x \asymp g_\lambda'(x) - 1 \asymp g_\lambda''(x) \asymp \lambda.
  \end{equation*}
  That $|\phi_\eta(n/N) - \phi_\eta(m/N) - \lambda| \le  \delta  $ implies $n = Ng_\lambda(m/N) + O(\delta N) $ by monotonicity, so
  \begin{equation*}
    |\mathcal{S}|\ll \sum_{m\sim N } \bbone_{  \| N g_\lambda(m/N)  \|\ll \delta N}
    =  \sum_{m\sim N } \bbone_{  \| N (g_\lambda(m/N) - m/N)  \|\ll \delta N}.
  \end{equation*}

  Note that $N(g_\lambda(m/N) - m/N) \asymp N\lambda$ and the corresponding values are spaced by $\gg \lambda$ for $m \sim N$. For $\lambda\le N^{-1/3} $, the desired result follows since
  \begin{equation*}
    |\mathcal{S}| 
    \ll (1 + N\lambda) \biggl( 1 + \frac{\delta N}{\lambda} \biggr)
    \ll \delta N^2 + \lambda N \ll \delta N^2 + N^{2/3}.
  \end{equation*}
  
  Suppose that $\lambda > N^{-1/3} $ from now on. By \cite[Proposition 3.3]{2024arXiv2401.13981}
  (with $\delta $ in the reference taken to be $\delta_1 = \delta N+ (\lambda/N)^{1/3} $ by positivity) we get
  \begin{equation*}
    \sum_{m\sim N } \bbone_{  \| N (g_\lambda(m/N) - m/N)  \|\ll \delta N}
    \ll \delta N^2 + (\lambda/N)^{1/3}N + \sqrt{\frac{N\lambda}{\delta_1}} + \sqrt{\frac{N\delta_1}{\lambda}}.
  \end{equation*}
  It remains to check that the bound is acceptable. Note from $\lambda \ll 1$ that
  \begin{equation*}
    \sqrt{\frac{N\lambda}{\delta_1}} \ll \sqrt{\frac{N\lambda}{(\lambda/N)^{1/3}}} \ll N^{2/3} 
  \end{equation*}
  and from $\delta \ll 1/N$ and $\lambda > N^{-1/3}$ that
  \begin{equation*}
    \sqrt{\frac{N\delta_1}{\lambda}} \ll \sqrt{\frac{\delta N^2}{\lambda}} + \sqrt{\frac{N^{2/3}}{\lambda^{2/3}}}
    \ll \sqrt{N/\lambda} + (N/\lambda)^{1/3}\ll N^{2/3},
  \end{equation*}
  from which the desired result follows.
\end{proof}

\begin{lemma}[Smoothing]\label{lem:F_smoothing}
  Let $A \geq 1$ and $\mathcal{S} \subset [-A, A]$ be a finite set. For $\alpha \in \mathbb{R}$, denote
  \begin{equation*}
    C(\alpha) := \sum_{s\in \mathcal{S}} c_s e(s\alpha),
  \end{equation*}
  where $c_s\in \mathbb{C}$. Then for any $p, \ell \geq 1$ we have
  \begin{equation*}
    |C(\alpha)|^p \ll_{p, \ell} A  \int_\mathbb{R} |C(\xi)|^p \cdot \big(1+A|\xi-\alpha|\big)^{-\ell}  \, d\xi.
  \end{equation*}
\end{lemma}

\begin{proof}
  Let $\psi$ be a Schwartz function with $\psi(x) = 1$ for $x\in [-1, 1]$. For $f(s) = \psi(s/A) e(s\alpha)$, note that $\widehat{f}(\xi) = A\widehat{\psi}(A(\xi-\alpha))$.
  Hence for any $\ell\geq 1$, the decay of $\widehat{\psi}$ gives
  \begin{equation}
    |C(\alpha)| = \Big|\sum_{s \in \mathcal{S}}c_s \psi\Big(\frac{s}{A}\Big) e(s\alpha)\Big| = \Big|\sum_{s\in \mathcal{S}} c_s \int_{-\infty}^\infty A\widehat{\psi}(A(\xi-\alpha)) e(\xi s) \, d\xi\Big| 
    \ll_\ell A \int_{-\infty}^\infty  \frac{|C(\xi)|}{\big(1+A|\xi-\alpha|\big)^{\ell} } \, d\xi.
  \end{equation}

  This settles the case $p=1$, and the general case follows from H\"older's inequality.
\end{proof}

\begin{lemma}[Fourier bound]\label{lemma:energy_F_bound}
  For any integer $k\geq 1$ and real $1\le d\le W$, we have 
  \begin{equation*}
    V^2E_{2k}(\mathcal W;d) \prec_k d \sum_{m, n\sim N \\ |\xi_m-\xi_n| \prec \frac{1}{d}} |F(\xi_m-\xi_n)|^{2k-1}.
  \end{equation*}
\end{lemma}

\begin{proof}
  Note that 
  \begin{equation}\label{eq:energy_prelim}
    V^2E_{2k}(\mathcal W;d)\leq
    \sum_{w_1,\ldots,w_{2k}\in\mathcal W}
    |D_\eta(w_{2k})|^2\cdot 
    \bbone\Big(\Big|\sum_{i=1}^k w_i - \sum_{i = k+1}^{2k} w_i\Big|\leq d\Big).
  \end{equation}
  By Lemma \ref{lem:F_smoothing} applied to $D_\eta(\alpha) = \sum_{n\sim N} a_n e(\alpha \xi_n)$, 
  \begin{align*}
    |D_\eta(\alpha)|^2 \ll_\ell  \int_\mathbb{R} |D_\eta(\xi)|^2 \cdot (1+|\xi-\alpha|)^{-\ell} \, d\xi \prec_\ell \int_{|\xi-\alpha| \prec 1} |D_\eta(\xi)|^2 \, d\xi + W^{-\ell}.
  \end{align*}
  Thus for any $\tau \in \mathbb{R}$, since $\mathcal{W}$ is $1$-spaced, we have
  \begin{equation}
    \sum_{w\in \mathcal{W}\\|\tau - w|\leq d} |D_\eta(w)|^2 \prec_\ell \sum_{w\in \mathcal{W} \\ |w-\tau| \leq d} \int_{|\xi - w| \prec 1} |D_\eta(\xi)|^2 \, d\xi + W^{-\ell} 
    \prec \int_{|\xi - \tau| \prec d} |D_\eta(\xi)|^2 \, d\xi + W^{-\ell}.
  \end{equation}
  Taking $\tau = w_1 + \cdots + w_k - w_{k+1} - \cdots - w_{2k-1}$, by \eqref{eq:energy_prelim} we obtain
  \begin{align*}
    V^2E_{2k}(\mathcal W;d)
    &\prec_\ell \sum_{w_1,\ldots,w_{2k-1}\in\mathcal W} \int_{|s| \prec d} |D_\eta(\tau - s)|^2 \, ds + W^{2k-\ell}.
  \end{align*}
  We may ignore the error term $W^{2k-\ell}$ by taking $\ell$ sufficiently large in terms of $k$ and comparing with the diagonal in the left-hand side. The condition $|s| \prec d$ can be replaced by a non-negative bump function $g(s/d)$ by positivity, with $\widehat{g}(x) \prec_j (1+|x|)^{-j}$. Opening the square yields that
  \begin{align*}
    V^2 E_{2k}(\mathcal{W}; d)&\prec_k \sum_{m, n\sim N} a_m \overline{a_n} \int_{\mathbb{R}} F(\xi_m-\xi_n) |F(\xi_m-\xi_n)|^{2(k-1)} g\Big(\frac{s}{d}\Big)e(-s(\xi_m-\xi_n)) \, ds \\
    &\prec_k d \sum_{m, n\sim N} |F(\xi_m-\xi_n)|^{2k-1} \big|\widehat{g}\big(d(\xi_m-\xi_n)\big)\big|.
  \end{align*}
  Here we made crucial use of the assumption $|a_n| \leq 1$. Then the decay of $\widehat{g}$ and a trivial bound for $F$ lead to
  \begin{equation*}
    V^2E_{2k}(\mathcal W;d) \prec_k d \sum_{m, n\sim N \\ |\xi_m-\xi_n| \prec \frac{1}{d}} |F(\xi_m-\xi_n)|^{2k-1} + d W^{-100k}.
  \end{equation*}
  The term $dW^{-100k}$ is smaller than the left-hand side, so the result follows.
\end{proof}

\begin{proof}[Proof of Proposition \ref{lem:energy_bound}]
  Separating the diagonal terms $m=n$ in Lemma \ref{lemma:energy_F_bound} gives
  \begin{equation}\label{eq:energy_diagonal_split}
    V^2E_{2k}(\mathcal W;D)
    \prec_k
    N D|\mathcal W|^{2k-1} + D\sum_{m,n\sim N\\ m\ne n\\  |\xi_m-\xi_n|\prec \frac{1}{D}} |F(\xi_m-\xi_n)|^{2k-1}.
  \end{equation}

  By Lemma \ref{lemma:frequency_count}, since $W\leq N^{4/3}$ we have 
  \begin{equation*}
    N_W(\lambda) := \#\big\{ m, n \sim N : m\ne n,\  |\xi_m-\xi_n-\lambda|\prec 1/W \big\}
    \prec \frac{N^2}{W}+N^{2/3}
    \prec \frac{N^2}{W}.
  \end{equation*}
  From $D\leq W$, the conditions $|\xi_m-\xi_n|\prec 1/D$ and $|\xi_m-\xi_n - \lambda|\prec 1/W$ imply $|\lambda|\prec 1/D$. Applying Lemma \ref{lem:F_smoothing} to $|F|^{2k-1}$ and truncating the integrals using $|F| \leq |\mathcal{W}|\ll W$, we get
  \begin{align}
    D\sum_{m,n\sim N\\ m\ne n\\  |\xi_m-\xi_n|\prec \frac{1}{D}} |F(\xi_m-\xi_n)|^{2k-1} &\prec_k DW\int_{|\lambda|\prec \frac{1}{D}}|F(\lambda)|^{2k-1}
    N_W(\lambda) \,d\lambda + W^{-100} \nonumber \\
    &\prec_k DN^2\int_{|\lambda|\prec \frac{1}{D}} |F(\lambda)|^{2k-1}\,d\lambda + W^{-100}. \label{eq:energy_integral_bound}
  \end{align}
  The error term $W^{-100}$ can be absorbed into the diagonal contribution, so we may ignore it.

  Now observe for any integer $j \geq 1$ that
  \begin{equation}\label{eq:energy_fourier_moment}
    D\int_{|\lambda|\prec \frac{1}{D}} |F(\lambda)|^{2j} \, d\lambda \prec E_{2j}(\mathcal W;D).
  \end{equation}
  Indeed, the condition $|\lambda| \prec \frac{1}{D}$ can be replaced by a non-negative bump function $g(D\lambda/L)$ by positivity, with $1\leq L \prec 1$ and  $\widehat{g}$ bounded and supported in $[-1, 1]$. Hence
  \begin{equation}
    D\int_{|\lambda| \prec \frac{1}{D}} |F(\lambda)|^{2j} \,d\lambda \leq D \int_{\mathbb R} g\Big(\frac{D\lambda}{L}\Big) |F(\lambda)|^{2j} \,d\lambda
    = \ L \sum_{w_1,\ldots,w_{2j}\in\mathcal W} \widehat{g}
    \Big( \frac{L}{D}
    \Big(\sum_{i=1}^j w_i-\sum_{i=j+1}^{2j} w_i\Big) \Big) ,
  \end{equation}
  which is $O(L E_{2j}(\mathcal W; D/L)) $. This yields \eqref{eq:energy_fourier_moment} by positivity, as $L\geq 1$.

  Now observe from H\"older's inequality followed by
  \eqref{eq:energy_fourier_moment} that
  \begin{align*}
    D\int_{|\lambda|\prec \frac{1}{D}} |F(\lambda)|^{2k-1}\, d\lambda
    &\leq \Big( D\int_{|\lambda|\prec \frac{1}{D}}|F(\lambda)|^2\,d\lambda \Big)^{\frac{1}{2k-2}} \Big( D\int_{|\lambda|\prec \frac{1}{D}}|F(\lambda)|^{2k}\,d\lambda \Big)^{\frac{2k-3}{2k-2}}\\
    &\prec E_2(\mathcal W;D)^{\frac{1}{2k-2}} \cdot E_{2k}(\mathcal W;D)^{\frac{2k-3}{2k-2}}.
  \end{align*}
  Inserting this into \eqref{eq:energy_integral_bound} and \eqref{eq:energy_diagonal_split}, we obtain
  \begin{align*}
    V^2E_{2k}(\mathcal W;D)
    \prec_k ND|\mathcal W|^{2k-1} + N^2 E_2(\mathcal W; D)^{\frac{1}{2k-2}} \cdot  E_{2k}(\mathcal W; D)^{\frac{2k-3}{2k-2}}.
  \end{align*}
  If the term $ND|\mathcal W|^{2k-1}$ dominates, then we are done. Otherwise, we obtain
  \begin{equation*}
    E_{2k}(\mathcal W;D) \prec_k \Big(\frac{N^2}{V^2}\Big)^{2k-2} E_2(\mathcal W;D),
  \end{equation*}
  as desired.
\end{proof}


\section{General large value bound}

Corollary \ref{theorem:curtz38o04} will follow from the more flexible large value bound below.

\begin{theorem}[General large value bound]\label{theorem:curwpy4v0k}
  For every integer $k \geq 2$, there exists $c_k > 0 $ such that the following holds. Let $N, V, W\ge 1$ satisfy
  \begin{equation*}
    N\le W\le N^{4/3} \quad \text{and} \quad  |\eta|  \le c_k W/N^2.
  \end{equation*}
  Let $a_n \in \mathbb{C}$ with $|a_n|\leq 1$, and denote as before $\Upsilon = 1 + |\eta|W$ and
  \begin{equation*}
    D_\eta(w) := \sum_{N < n\le 2N } a_n e(w\phi_\eta(n/N)) \qquad \text{for} \qquad \phi_\eta(x) := \sqrt{x} + \eta x^{3/2}.
  \end{equation*}
  Then for any $1$-spaced subset $\mathcal{W}$ of an interval of length $W$ such that $|D_\eta(w)|\ge V $ for all $w\in \mathcal{W}$, we have 
  \begin{align*}
    |\mathcal{W}| \prec_{k} \frac{N^2}{V^2} + \frac{N^{\frac{6k - 4}{2k + 1}}W^{\frac{k + 1}{2k + 1}}}{V^{\frac{8k - 4}{2k + 1}}} \left( \sup_{1\le D\ll_k  \Upsilon} \frac{E_2( \mathcal{W}; D)}{\sqrt{D}} \right)^{\frac{1}{2k + 1}} + \frac{N^{k + \frac{1}{2}} W^{\frac{k + 1}{2}} \Upsilon^{\frac{1}{4}}}{V^{2k + 1}}.
  \end{align*}
\end{theorem}

\begin{proof}[Proof of Corollary \ref{theorem:curtz38o04} assuming Theorem \ref{theorem:curwpy4v0k}]
  Applying Theorem \ref{theorem:curwpy4v0k} with $k = 2$ together with the bound $E_2(\mathcal{W}; D)\le (2D + 1)|\mathcal{W}|$, which follows from the $1$-spacing of
  $\mathcal{W}$, we obtain 
  \begin{equation*}
    |\mathcal{W}| \prec \frac{N^2}{V^2} + \frac{N^{\frac{8}{5}}W^{\frac{3}{5}}}{V^{\frac{12}{5}}}|\mathcal{W}|^{\frac{1}{5}}
    \Upsilon^{\frac{1}{10}} + \frac{N^{\frac{5}{2}}W^{\frac{3}{2}}}{V^5}\Upsilon^{\frac{1}{4}}.
  \end{equation*}
  Rearranging, this gives the desired bound
  \begin{equation*}
    |\mathcal{W}| \prec \frac{N^2}{V^2} +  \frac{N^2W^{\frac{3}{4}}\Upsilon^{\frac{1}{8}}}{V^3}
    + \frac{N^{\frac{5}{2}}W^{\frac{3}{2}}\Upsilon^{\frac{1}{4}}}{V^5}.
  \end{equation*}
\end{proof}


\section{Proof of Theorem \ref{theorem:curwpy4v0k}}

We may, by translation, reduce to the case of $\mathcal{W} \subset [W, 2W] $, which we suppose from now
on.

The proof starts with an application of duality. Let $\gamma_w = |D_\eta(w)|/ D_\eta(w) $. Then, by Cauchy--Schwarz,
we have that
\begin{equation}
  R^2 V^2\le \biggl| \sum_{n\sim N \\ w\in \mathcal{W} } a_n\gamma_w e(w \phi_\eta(n/N) ) \biggr|^2
  \ll N \sum_{n\sim N } \biggl| \sum_{w\in \mathcal{W} } \gamma_w e(w\phi_\eta(n/N)) \biggr|^2. 
\end{equation}
Fix some smooth $\bbone_{ [1, 2] }\le \varphi \le \bbone_{ [3/4, 9/4] }$. Then, by positivity and an application of
Poisson summation, the right-hand side of the above is at most
\begin{equation}\label{align:curw4688f1}
  N \sum_{n } \varphi \biggl(\frac{n}{N}\biggr) \biggl| \sum_{w\in \mathcal{W} } \gamma_w e(w\phi_\eta(n/N)) \biggr|^2 
  \, \,= N^2 \sum_{w, w'\in \mathcal{W} } \gamma_w \overline{\gamma_{w'}}
    \sum_{m } \int_{\mathbb{R}} e((w - w') \phi_\eta(\xi) -  Nm\xi)\varphi(\xi) \, d \xi.
\end{equation}
By repeated integration by parts and $1$-spacing of $\mathcal{W}$, the contribution of $m = 0 $ is
\begin{equation}
  \ll N^2  \sum_{w, w'\in \mathcal{W} }  (1 + |w - w'|)^{-10} \ll N^2R.
\end{equation}

We now prepare for our application of stationary phase. A stationary point $\xi_0$ satisfies
\begin{equation*}
    (w-w')\Big(\frac{1}{2\sqrt{\xi_0}} + \frac{3\eta\sqrt{\xi_0}}{2} \Big) = Nm.
\end{equation*}
Note that $\xi_0$ is in the support of $\varphi$ (so that $\xi_0 \in \big[\frac{3}{4}, \frac{9}{4}\big]$) only if (say) $\frac{w-w'}{Nm}\in \big[\frac{3}{2}, \frac{7}{2}\big]$. Such a $\xi_0$ is unique and satisfies
\begin{equation}
  \xi_0 = \Big(\frac{w - w'}{Nm}\Big)^2  \biggl(1 + \sqrt{1 - 3\eta\Big(\frac{w - w'}{Nm}\Big)^2} \biggr)^{-2}.
\end{equation}
For $\frac{w-w'}{Nm}\in \big[\frac{3}{2}, \frac{7}{2}\big]$, note that $N|m| \leq \frac{2W}{3}$ and Taylor expansion at $\eta = 0 $ yields that
\begin{equation}
  (w - w')\phi_\eta(\xi_0) -  Nm\xi_0 =
  \frac{(w - w')^2}{4Nm} + \eta\frac{(w - w')^4}{8N^3m^3} + \eta^2 Nm \Psi_\eta \biggl(\frac{w - w'}{Nm}\biggr)
\end{equation}
for some $\Psi_\eta \in C_c^\infty([1, 4])$ with $|\Psi_\eta^{(j)}(x)| \ll_j 1 $.
Applying stationary phase in the form of \cite[\S3, Main Theorem]{MR3994723}, and incorporating $e\big(\eta^2Nm \Psi_\eta(\frac{w-w'}{Nm})\big)$ into the weight function, we have that \eqref{align:curw4688f1}
equals
\begin{equation}\label{eq:curw7v18du}
  N^2 \sum_{w, w'\in \mathcal{W} } \gamma_w  \overline{\gamma_{w'}} \sum_{m \neq 0 } \frac{1}{\sqrt{N|m|}}
  \Phi_\eta \biggl(\frac{w - w'}{Nm}, Nm\biggr)
  e \biggl( \frac{(w - w')^2}{4Nm} + \eta \frac{(w - w')^4}{8N^3m^3} \biggr)
  + O(N^2R)
\end{equation}
for some smooth $\Phi_\eta $ supported in $[1, 4]\times [-W, W] $ satisfying $|y^k \partial_1^j\partial_2^k \Phi_\eta(x, y) | \ll_{j, k} 1$
for $\eta^2 |y|\ll 1 $. This last bound applies in the whole support of $\Phi_\eta$, since $\eta^2 W \leq c_k^2 W^3/N^4 \leq c_k^2$.

We shall now bound the contribution of $|m|\sim M$ for any scale $M \ll W/N$.
Note that this forces $|w - w'|\asymp MN $. Partitioning $\mathcal{W} $ into intervals
$(\mathcal{W}_j)_{j\ll W/(MN)} $ of length $\sim C MN $ for $C>0$ sufficiently large, we get 
\begin{align}
  &\biggl|N^2 \sum_{w, w'\in \mathcal{W} } \gamma_w  \overline{\gamma_{w'}} \sum_{|m|\sim M} \frac{1}{\sqrt{N|m|}}
  \Phi_\eta \biggl(\frac{w - w'}{Nm}, Nm\biggr)
  e \biggl( \frac{(w - w')^2}{4Nm} + \eta \frac{(w - w')^4}{8N^3m^3} \biggr)\biggr| \\ 
  \ll \ & \frac{N^{3/2}}{M^{1/2}} \sum_{|m| \sim M }  \sum_{j\ll \frac{W}{MN} \\ -1\le h\le 1 }
  \biggl| \sum_{w \in \mathcal{W}_j \\ w'\in \mathcal{W}_{j + h } } \gamma_w \overline{\gamma_{w'}} \Phi_\eta \biggl(\frac{w - w'}{Nm}, Nm\biggr) e \biggl( \frac{(w - w')^2}{4Nm} + \eta \frac{(w - w')^4}{8N^3m^3} \biggr)\biggr|.
\end{align}
By H\"older's inequality, this is
\begin{align}\label{eq:interm_after_holder}
  \ll \frac{N^{3/2}}{M^{1/2}}
   \sum_{j\ll \frac{W}{MN} \\ -1\le h\le 1 } & (M |\mathcal{W}_j|)^{1- \frac{1}{2k}} \\
   \times &\ \biggl(  \sum_{ |m| \sim M } \sum_{w\in \mathcal{W}_{j} } \biggl| \sum_{w'\in \mathcal{W}_{j + h} } \overline{\gamma_{w'}}\Phi_\eta \biggl(\frac{w - w'}{Nm}, Nm\biggr) e \biggl( \frac{(w - w')^2}{4Nm} + \eta \frac{(w - w')^4}{8N^3m^3} \biggr)  \biggr|^{2k} \biggr)^{\frac{1}{2k}}. 
\end{align}
Denoting $u_j := \min(\mathcal{W}_j) - 1/3 \asymp W$, by Lemma \ref{lemma:curw6ekln5} we have 
\begin{align}
  &\sum_{w\in \mathcal{W}_{j} } \biggl| \sum_{w'\in \mathcal{W}_{j + h} } \overline{\gamma_{w'}} \Phi_\eta \biggl( \frac{w - w'}{Nm}, Nm \biggr) e \biggl( \frac{(w - w')^2}{4Nm} + \eta \frac{(w - w')^4}{8N^3m^3} \biggr)  \biggr|^{2k}\\ 
  \ll_k &\
  MN \int_\mathbb{R} \chi(x) \biggl| \sum_{w'\in \mathcal{W}_{j + h} - u_j } \overline{\gamma_{w' + u_j}} \Phi \biggl(\frac{Nmx - w'}{Nm}, Nm\biggr) 
   e \biggl( \frac{(Nm x - w')^2}{4Nm} + \eta \frac{(Nmx - w')^4}{8N^3m^3} \biggr)  \biggr|^{2k} \,d x
\end{align}
for some fixed non-negative $\chi\in C_c^\infty(\mathbb{R}) $ and some ${\Phi\in \{ \Phi_\eta, \Phi_\eta^*\} } $, where 
\begin{equation*}
    \Phi^*_\eta(x, y) := y^{-1}  \partial_1 \Phi_\eta(x, y) + \pi i (x + \eta x^3)\Phi_\eta(x, y).
\end{equation*}
Either way, $|y^k\partial_1^j \partial_2^k \Phi(x, y)| \ll_{j, k} 1$. Expanding the $2k $-th power, we get that \eqref{eq:interm_after_holder} is
\begin{equation}\label{eq:curw6ikfwz}
  \ll_k N^{\frac{3}{2} + \frac{1}{2k}}M^{\frac{1}{2}}  \sum_{j\ll \frac{W}{MN} \\  -1\le h\le 1 } |\mathcal{W}_j|^{1 - \frac{1}{2k}}
  \biggl(\sum_{|m| \sim M } \sum_{\mathbf{w} \in (\mathcal{W}_{j + h} - u_j)^{2k} }
  |I_{m}(\mathbf{w})|\biggr)^{\frac{1}{2k}},
\end{equation}
where denoting $\mathbf{w} = (w_1, \dots, w_{2k})$ and $\Delta_\ell := w_1^{\ell} + \dots + w_k^\ell - w_{k + 1}^\ell - \dots - w_{2k}^\ell$, we put\footnote{Note that by putting absolute values around the integral, we forfeit cancellation over the $m$-sum. It should be possible to use the oscillation of this sum to sharpen our results, but we refrain from doing so for simplicity.}
\begin{align}
  I_m( \mathbf{w} ) := \int_\mathbb{R} \chi(x) & \Bigg[\prod_{i=1}^k \Phi \biggl( x-\frac{w_i}{Nm} , Nm \biggr)\overline{\Phi \biggl( x-\frac{w_{i + k}}{Nm}, Nm \biggr)}\Bigg] \\
  \times \ & e \biggl( -\frac{1}{2}\biggl( \Delta_1 + \frac{\eta \Delta_3}{N^2m^2} \biggr)x +
  \frac{3\eta \Delta_2}{4Nm} x^2 - \frac{\eta\Delta_1}{2}x^3\biggr) \,d x.
\end{align}

Note that $|w_i| \ll MN$. Using repeated integration by parts \cite[Lemma 3.1]{MR3994723} when $|\Delta_1|\geq C_k \Upsilon$ for $C_k > 0$ sufficiently large, and the oscillatory integral
bound of our Lemma \ref{lemma:curyjy52ai} otherwise (absorbing the cubic term $e\big(-\frac{\eta \Delta_1}{2}x^3\big)$ into the smoothing via $|\eta \Delta_1| \ll W^3/N^4 \ll 1$), we have 
\begin{equation}\label{eq:2}
  I_m(\mathbf{w}) \ll_k \min\biggl(\max(1, |\Delta_1|- C_k\Upsilon)^{-100},
  \biggl( 1 + \biggl| \Delta_1 + \frac{\eta\Delta_3}{N^2m^2} \biggr| + \biggl| \frac{\eta\Delta_2}{NM} \biggr| \biggr)^{-1/2}
  \biggr).
\end{equation}
At this point, we use the average over $m $.
If $1\ll |\Delta_1|\ll_k \Upsilon $, then either $\alpha_\mathbf{w}(m) := \Delta_1 + \frac{\eta\Delta_3}{N^2m^2}$ satisfies $|\alpha_\mathbf{w}(m)| \geq \frac{|\Delta_1|}{2}$ for all $|m|\sim M$, or it satisfies $|\alpha_\mathbf{w}'(m)| \asymp \frac{|\Delta_1|}{M}$ for all $|m| \sim M$. Thus
\begin{equation}
  \sum_{|m| \sim M } |I_m(\mathbf{w})|\ll_k \sum_{|m| \sim M } (1 + |\alpha_\mathbf{w}(m)|)^{-1/2} \ll \frac{M}{\sqrt{|\Delta_1|}} +  1 +\int_{0}^{M} \Big(1 + \frac{|\Delta_1|}{M}u\Big)^{-1/2} \,d u
  \ll 1 + \frac{M}{\sqrt{|\Delta_1|}}.
\end{equation}
Applying the trivial bound $O(M) $ when $|\Delta_1| \ll 1 $ and the decay when $|\Delta_1| \geq 2C_k \Upsilon$, we get 
\begin{equation}
  \sum_{|m| \sim M } |I_m(\mathbf{w})|\ll_k 1 + \frac{M}{\sqrt{1 + |\Delta_1|}} \qquad \text{for all } \mathbf{w} \in (\mathcal{W}_{j + h} - u_j)^{2k}.
\end{equation}

Splitting $\Delta_1 $ into the ranges $|\Delta_1| \leq 1 $ and $|\Delta_1| \sim D = 2^d$ for $d\geq 0$, and observing that the shift by $u_j$ can be ignored in $\Delta_1$, we have that \eqref{eq:curw6ikfwz} is 
\begin{equation}\label{eq:curw7voo8y}
  \ll_k N^{\frac{3}{2} + \frac{1}{2k}} M^{\frac{1}{2}}
  \sum_{j\ll \frac{W}{MN} \\ -1\le h\le 1 } |\mathcal{W}_j|^{1 - \frac{1}{2k}}
  \biggl( \sum_{ 1\le D\ll_k \Upsilon \\ D = 2^d } \biggl(1 + \frac{M}{\sqrt{D}}\biggr) E_{2k}(\mathcal{W}_{j + h}; D) \biggr)^\frac{1}{2k}.
\end{equation}
By H\"older's inequality, using that
\begin{equation}
  \sum_{j } |\mathcal{W}_j| = R \qquad \text{and} \qquad \sum_{j } E_{2k}(\mathcal{W}_{j + h} ; D) \le  E_{2k}(\mathcal{W}; D),
\end{equation}
we obtain that \eqref{eq:curw7voo8y} is
\begin{equation}
  \ll N^{\frac{3}{2} + \frac{1}{2k}}M^{\frac{1}{2}} R^{1 - \frac{1}{2k}}
  \biggl( \sum_{ 1\le D\ll_k \Upsilon \\ D = 2^d} \biggl(1 + \frac{M}{\sqrt{D}}\biggr) E_{2k}(\mathcal{W}; D)  \biggr)^{\frac{1}{2k}}.
\end{equation}
Summing over dyadic scales $M\ll W/N$, using $\Upsilon \ll_k (W/N)^2$, and collecting this section's bounds, we conclude that 
\begin{equation}\label{eq:curw7wcn6k}
  R^2 V^2 \ll_k   N^2R + N W^{\frac{1}{2} + \frac{1}{2k}} R^{1 - \frac{1}{2k}}
  \biggl( \sum_{1\le D\ll_k \Upsilon \\ D = 2^d} D^{-1/2} E_{2k}(\mathcal{W}; D)\biggr)^{\frac{1}{2k}}.
\end{equation}

At this point, we use the energy bound in Proposition \ref{lem:energy_bound}, which states that
\begin{equation}\label{eq:curw7wcd13}
  E_{2k}(\mathcal{W}; D) \prec_k \Big(\frac{N}{V}\Big)^{4k - 4} E_2(\mathcal{W}; D)
  + \frac{ND}{V^2} R^{2k - 1}.
\end{equation}
Inserting \eqref{eq:curw7wcd13} into \eqref{eq:curw7wcn6k} and rearranging for a bound on
$R $ yields Theorem \ref{theorem:curwpy4v0k}, as desired.


\section{Atkinson's formula}

In \S\ref{sec:curvbfnzzp}, our passage from large values of $|\zeta(1/2 + it)|$ to exponential sums is via Atkinson's formula \cite{MR31963} for the second moment in short intervals. We state it below in the greatly simplified upper bound version due to Heath-Brown \cite[Lemma 1]{MR517737}, whose proof was simplified by Ivi\'c \cite{MR1994094} 

\begin{lemma}[Atkinson's formula {\cite[Theorem 7.2]{MR1994094}}]\label{lem:atkinson}
  Let $\varepsilon > 0$ and suppose that $T\geq 2$ and $T^\varepsilon \le G\le T^{1/2-\varepsilon}$. Denote
  \begin{equation}\label{eq:S_T_def}
    S_T(\Delta, N) := \sum_{N \leq n \leq N+\Delta} (-1)^n d(n) e(f(T, n))
  \end{equation}
  for
  \begin{equation}\label{eq:f_T_n_def}
    f(T,n) := \frac{T}{\pi} \left(\sinh^{-1}\left(\sqrt{\frac{\pi n}{2T}} \right) + \sqrt{\frac{\pi n}{2 T}\Big(1 + \frac{\pi n}{2T}\Big)}\right).
  \end{equation}
  Then
  \begin{align*}
    &\int_{T-G}^{T+G} |\zeta(\tfrac{1}{2}+it)|^2 \, dt  \\
    & \qquad \ll_{\varepsilon} G\log T + G \sum_{\substack{T^{1/3} \leq N \leq \frac{T}{G^2}(\log{T})^{1+\varepsilon} \\ N = 2^m}} (TN)^{-1/4} \Bigg(|S_T(N, N)| + \int_{0}^N \frac{|S_T(\Delta, N)|}{N}\, d\Delta \Bigg).
  \end{align*}
\end{lemma}

Truncating the power series expansion of the phases above and splitting into small intervals, we bound the short second moment of $|\zeta(1/2 + it)|$ in terms of exponential sums with the special phases that we are considering.

\begin{corollary}[Phase truncation]\label{cor:atkinson-model}
  Let $1\leq N\leq T^{3/5}$,
  $1\leq W\leq\sqrt{TN}$, and $N^{1/2}W^2\leq T^{3/2}$. For each $t_0\in[T,2T]$, there exist coefficients $a_n \in \mathbb{C}$ with $|a_n|\leq 1$ such that for any $t \in \big[t_0, t_0 + W(T/N)^{1/2}\big]$ and $\varepsilon > 0$ we have
  \begin{equation*}
    \max_{0\leq \Delta\leq N} |S_t(\Delta,N)| \ll_\varepsilon N^\varepsilon \max_{0 \leq \Delta \leq N} \biggl|\sum_{N \leq n \leq N + \Delta} a_n e \Big(w \Big[\sqrt{\frac{n}{N}} -\frac{\pi N}{12 t_0}\Big(\frac{n}{N}\Big)^{3/2}\Big]\Big) \biggr|,
  \end{equation*}
  where $S_t(\Delta,N)$ is given in \eqref{eq:S_T_def} and
  \begin{equation*}
    w := \sqrt{\frac{2N}{\pi}}\bigl(\sqrt t-\sqrt{t_0}\bigr) \in[0,W].
  \end{equation*}
\end{corollary}

\begin{proof}
  We may assume that $T$ is sufficiently large, otherwise the result is trivial. For $|z|\leq \frac{1}{2}$, a Taylor expansion gives
  \begin{equation*}
    g(z) := \big(\sinh^{-1}(z)+z\sqrt{1+z^2}\big) - \Big(2z+\frac{z^3}{3}\Big) = O(|z|^5)
  \end{equation*}
  and $g'(z) = O(|z|^4)$. Note that $t \asymp T$ since $W \leq \sqrt{TN}$, and let $N < n\leq 2N$. Taking $z=\sqrt{\frac{\pi n}{2t}} \asymp \sqrt{N/T}\ll T^{-1/5}$ in \eqref{eq:f_T_n_def}, we obtain
  \begin{equation*}
    f(t,n) =\sqrt{\frac{2nt}{\pi}} +\frac{\sqrt{2\pi}}{12}\frac{n^{3/2}}{\sqrt{t}} +r_t(n)
  \end{equation*}
  with $|r_t(n)| + n|r_t'(n)| \ll N^{5/2}T^{-3/2} \le 1$. Thus partial summation yields
  \begin{equation}\label{eq:atkinson-first-truncation}
    \max_{0\leq \Delta\leq N}|S_t(\Delta,N)|
    \ll\max_{0\leq \Delta\leq N}
    \biggl|\sum_{N \leq n\leq N+\Delta}(-1)^nd(n)
    e\biggl(\sqrt{\frac{2nt}{\pi}}
    +\frac{\sqrt{2\pi}}{12}\frac{n^{3/2}}{\sqrt{t}}\biggr)\biggr|.
  \end{equation}

  Next we replace $t$ with $t_0$ in the second term. Using
  \begin{equation*}
    \frac{1}{\sqrt t} - \frac1{\sqrt{t_0}} = -\frac{\sqrt t-\sqrt{t_0}}{t_0} +\frac{(\sqrt t-\sqrt{t_0})^2}{t_0\sqrt t},
  \end{equation*}
  the phase in \eqref{eq:atkinson-first-truncation} equals
  \begin{equation*}
    \sqrt{\frac{2nt_0}{\pi}} +\frac{\sqrt{2\pi}}{12}\frac{n^{3/2}}{\sqrt{t_0}} +w\Big[\sqrt{\frac{n}{N}} -\frac{\pi N}{12 t_0}\Big(\frac{n}{N}\Big)^{3/2}\Big] +q_t(n)
  \end{equation*}
  for 
  \begin{equation*}
    q_t(n) = \frac{\sqrt{2\pi}}{12} \frac{n^{3/2}(\sqrt t-\sqrt{t_0})^2}{t_0\sqrt t}.
  \end{equation*}
  Since $|t-t_0| \ll W\big(\frac{T}{N}\big)^{1/2}$ and $t \asymp T$, this gives
  \begin{equation*}
    |q_t(n)|+n|q_t'(n)|
    \ll \frac{N^{3/2}(\sqrt t-\sqrt{t_0})^2}{T^{3/2}}
    \ll \frac{N^{1/2}W^2}{T^{3/2}}\leq1.
  \end{equation*}
  Thus a second partial summation removes $e(q_t(n))$. Taking
  \begin{equation*}
    a_n = (-1)^n d(n) e\biggl(\sqrt{\frac{2nt_0}{\pi}} +\frac{\sqrt{2\pi}}{12}\frac{n^{3/2}}{\sqrt{t_0}}\biggr),
  \end{equation*}
  which does not depend on $t$, and rescaling via the divisor bound so that $|a_n|\leq 1$, we obtain the claimed result.
\end{proof}

The maxima in the result above were a convenient device to quickly execute partial summation, but they need to be removed once we sum over distinct points. This is accomplished at little loss by the simple lemma below.

\begin{lemma}[Removing maxima of partial sums]\label{lem:partial-sums}
  Let $N,J\geq1$, $\mathcal W_j\subset\mathbb R$
  be finite sets, $a_{j,n}\in\mathbb C$, and 
  $\xi_{j,n}\in\mathbb R$. Then there exist $a'_{j, n} \in \mathbb{C}$ with $|a'_{j, n}| = |a_{j, n}|$ such that
  \begin{equation}
    \sum_{1\leq j\leq J}\sum_{w\in\mathcal W_j}
      \max_{0\leq\Delta\leq N}
      \biggl|\sum_{N < n\leq N+\Delta}
      a_{j,n}e(w\xi_{j,n})\biggr|
    \ll \ \log(2N)
        \sum_{1\leq j\leq J}\sum_{w\in\mathcal W_j}
        \biggl|\sum_{n\sim N}a'_{j,n}
        e(w\xi_{j,n})\biggr|.
  \end{equation}
\end{lemma}

\begin{proof}
  We may assume that $N, \Delta \in \mathbb{Z}$. Summing the geometric series, note that the Fourier coefficients of the indicator of $\{1,\ldots,\Delta\}$ on $\mathbb Z/(2N)\mathbb Z$ are
  \begin{equation*}
    \widehat{\bbone}_{[1, \Delta]}(h) := \frac1{2N}\sum_{a=1}^{\Delta}
    e\Big(-\frac{ha}{2N}\Big)
    \ll\frac1{1+|h|}
  \end{equation*}
  for each $-N<h\leq N$. Applying the Fourier expansion at $n-N$ gives
  \begin{align*}
    &\max_{0\leq\Delta\leq N} \biggl|\sum_{N<n\leq N+\Delta} a_{j,n}e(w\xi_{j,n})\biggr| \ll \sum_{-N<h\leq N}\frac1{1+|h|} \biggl|\sum_{n\sim N} a_{j,n} e\Big(\frac{hn}{2N}\Big)e(w\xi_{j,n})\biggr|.
  \end{align*}
  Summing over $1\leq j\leq J$ and $w \in \mathcal{W}_j$, since
  $\sum_{-N<h\leq N}(1+|h|)^{-1}\ll\log(2N)$ we obtain the conclusion with $a'_{j, n} = a_{j, n} e\big(\frac{hn}{2N}\big)$ for some $-N < h \le N$.
\end{proof}

Let us also record the following standard result.

\begin{lemma}[{\cite[Theorem 1.2]{MR1230387}}]\label{lemma:curx2yevdr}
  For every $t \in \mathbb{R}$ and integer $k \ge 1$,
  \begin{equation*}
    |\zeta(\tfrac{1}{2} + it)|^k \ll_k 1 + \log(2 + |t|)  \int_{t - 1/3}^{t +  1/3} |\zeta(\tfrac{1}{2} + i u)|^k \,d u .
  \end{equation*}
\end{lemma}


\section{Refined spacing for large values of zeta} 

We will use the following result of Iwaniec \cite{zbMATH03683696} on
the short fourth moment of the Riemann zeta function on average. This is unnecessary if one only wishes to obtain some power saving for the sixth moment, but improves the exponents without much extra effort.

From now on, denote $f \prec g$ when $|f| \ll_\varepsilon T^\varepsilon |g|$ for any $\varepsilon > 0$.

\begin{proposition}[{\cite[Theorem 4]{zbMATH03683696}}]\label{prop:Iwaiec_short_fourth_moment}
  Let $T\geq 1$ and $\sqrt{T}\leq L \leq T$. Suppose that $I_1,\dots, I_J\subset [T, 2T] $ are disjoint intervals, each of length $L$. Then
  \begin{equation}
    \sum_{j\le J  } \int_{I_j} |\zeta(\tfrac{1}{2} + it)|^4 \,d t \prec JL + T\sqrt{J/L}.
  \end{equation}
\end{proposition}

As a corollary, we obtain the following upper bound on $E_2$ at any scale.

\begin{corollary}[$E_2$ refinement]\label{corollary:curxl2f9cb}
  Suppose that $1\le G\le T/10 $ and let $t_1,\dots, t_R\in [T, 2T] $ be $2G$-spaced points such that
  \begin{equation}
    \int_{|t - t_j|\le G} |\zeta(\tfrac{1}{2} + it)|^4 \,d t\ge V
  \end{equation}
  for each $1\leq j\le R$, where $V \geq 1$. Then for any $H\ge 1 $ we have 
  \begin{equation}
    E_2(\{t_1, \dots, t_R\}; H)\prec R + \frac{RH}{V} + \frac{T\sqrt{R}}{V^{3/2}}.
  \end{equation}
\end{corollary}

\begin{proof}[Proof]
  Since the intervals $[t_j-G,t_j+G]$ have disjoint interiors, the hypothesis and the fourth-moment bound give
  \begin{equation}\label{eq:global_mass}
    RV \leq \int_{T-G}^{2T+G} |\zeta(\tfrac12+it)|^4 \, dt \prec T.
  \end{equation}

  If $H<G$, then spacing gives $E_2(\{t_1, \ldots, t_R\}; H) = R$. Furthermore, if $H\geq T$, then a trivial bound and \eqref{eq:global_mass} imply
  \begin{equation*}
    E_2(\{t_1,\ldots,t_R\};H) \leq R^2 \prec \frac{RT}{V} \leq \frac{RH}{V}.
  \end{equation*}

  From now on, assume that $G \leq H < T$. Partition $\mathbb{R}$ into intervals $(I_\ell)_{\ell\in \mathbb{Z}}$ of length $H$, and let
  $m_\ell$ denote the number of points $t_j$ (for $1 \leq j \leq R$) which lie in $I_\ell$. Pairs $t_i, t_j$ with $|t_i-t_j|\le H$ lie in the same or adjacent intervals $I_\ell$, so
  \begin{equation*}
    E_2(\{t_1,\ldots,t_R\};H)\ll \sum_\ell m_\ell^2.
  \end{equation*}
  By dyadic pigeonholing, there exists a scale $1 \leq M \leq R$ such that denoting $J := |\mathcal J|$ for $\mathcal{J} := \{\ell \in \mathbb{Z} : M \leq m_\ell < 2M\}$, we have
  \begin{equation}\label{eq:chosen_occupancy}
    E_2(\{t_1, \ldots, t_R\}; H) \prec J M^2.
  \end{equation}
  Note that $JM \leq R$, and fix $\delta > 0$. If $MV\leq T^\delta\max(H,\sqrt T)$, then \eqref{eq:global_mass} gives
  \begin{equation}\label{eq:final_bound_small}
    JM^2\leq RM \ll \frac{RT^\delta}{V} \big(H +\sqrt{T}\big) \prec T^\delta\Big(\frac{RH}{V} + \frac{T\sqrt R}{V^{3/2}}\Big).
  \end{equation}

  If instead $MV >  T^\delta\max(H,\sqrt T)$, put
  \begin{equation*}
    L = \min(MVT^{-\delta}, T).
  \end{equation*}
  Then $L > \max(H,\sqrt T)$, since $H < T$, and clearly $L \leq T$. We have
  \begin{equation*}
    JMV \ll \sum_{\ell \in \mathcal{J}} \sum_{t_i \in I_\ell} \int_{|t-t_i| \leq G} |\zeta(1/2+it)|^4 \, dt.
  \end{equation*}
  Since $G\leq H< L \leq T$, the intervals of integration above can be covered by $O(J)$ disjoint intervals of length $L$. Applying Proposition \ref{prop:Iwaiec_short_fourth_moment}, we get
  \begin{equation*}
    JMV \prec JL + T \sqrt{J/L}.
  \end{equation*}
  The first term on the right-hand side is absorbed into the left-hand side, as $L \leq MVT^{-\delta}$. Thus
  \begin{equation*}
    JM^2V^2 \prec \frac{T^2}{L} \leq \frac{T^{2+\delta}}{MV} + T \prec \frac{T^{2+\delta}}{MV}.
  \end{equation*}
  Above we used that $MV \prec T$, due to \eqref{eq:global_mass}. Hence we conclude that
  \begin{equation*}
    JM^3 \prec \frac{T^{2+\delta}}{V^3}.
  \end{equation*}
  Since $JM\leq R$, this implies
  \begin{equation}\label{eq:final_bound_large}
    JM^2=\sqrt{JM}\cdot \sqrt{JM^3}
    \prec T^{\delta/2} \frac{T\sqrt R}{V^{3/2}}.
  \end{equation}

  Taking $\delta>0$ arbitrarily small and collecting \eqref{eq:chosen_occupancy}, \eqref{eq:final_bound_small}, and \eqref{eq:final_bound_large} yields the desired estimate
  \begin{equation*}
    E_2(\{t_1,\ldots,t_R\};H)
    \prec R+\frac{RH}{V}+\frac{T\sqrt R}{V^{3/2}}.
  \end{equation*}
\end{proof}


\section{Large value estimates for \texorpdfstring{$|\zeta(1/2 + it)|$}{}}\label{sec:curvbfnzzp}

In this section, we show the following general large value estimate, whose case $k = 2 $ corresponds to Theorem \ref{theorem:curvbfoast}. 

\begin{theorem}\label{theorem:curxme9sty}
    Let $E_2(\sigma) := \frac{24 - 128\sigma}{17}$ and  $[\sigma_-(2), \sigma_+(2)] := \big[\frac{7}{60}, \frac{5}{38}\big]$. For integers $k\ge 3$, let
    \begin{equation}
        E_k(\sigma) := \frac{32k^2 - 12k - 10 - (192k^2 - 120k - 28)\sigma}{16k^2 + 10k - 17 },
    \end{equation}
    \begin{equation}
        \sigma_-(k) := \frac{16k^2 - 22k + 7}{8( 16k^2 - 20k + 5)}, \qquad \text{and} \qquad
        \sigma_+(k) := \frac{4k - 3}{30k - 22}.
    \end{equation} 
    Then for any $\sigma\in [\sigma_-(k), \sigma_+(k)] $ we have 
    \begin{equation}
    N(T, T^\sigma) \prec_k T^{E_k(\sigma)}.
    \end{equation}
\end{theorem}

In Figure \ref{fig:large_val_plot}, we plot the exponent $E(\sigma)$ for the best bound $N(T, T^\sigma)\prec T^{E(\sigma)}$ obtained from Theorem \ref{theorem:curxme9sty} in each range, comparing it to the fourth and twelfth moment bounds.
\begin{figure}[htbp]
    \centering
    \resizebox{0.95\linewidth}{!}{
        \input{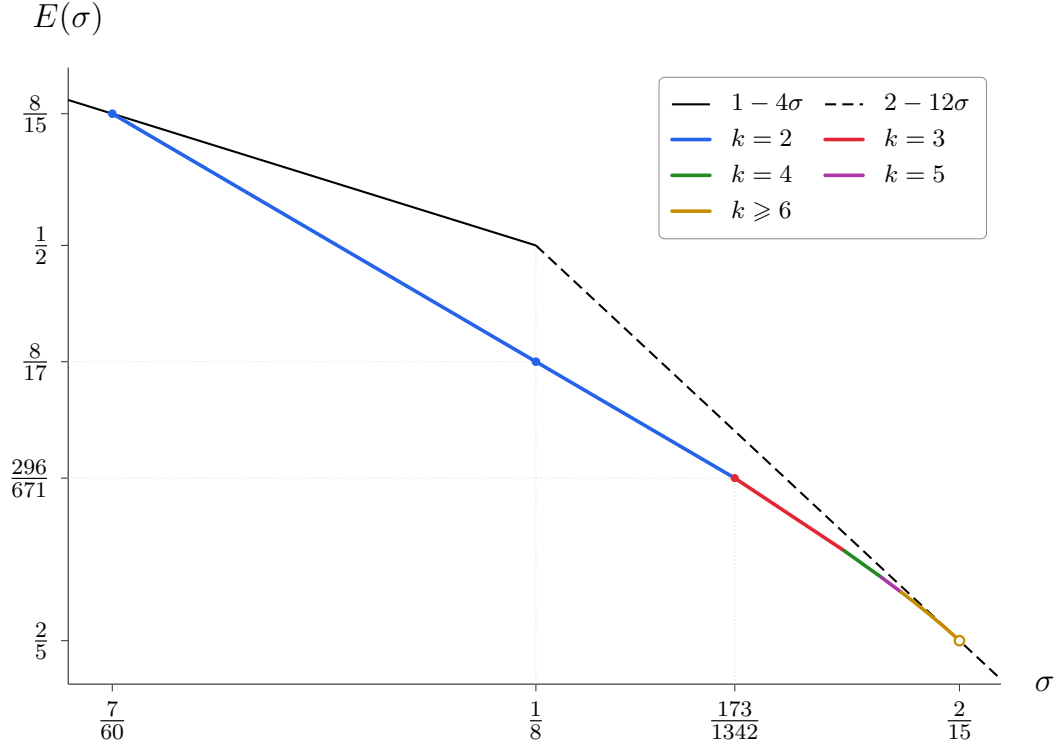}
    }
    \caption{Plot of new large value exponents} 
    \label{fig:large_val_plot}
\end{figure}

\begin{proof}[Proof of Theorem \ref{theorem:curxme9sty}]
  We may assume throughout the proof that $T$ is sufficiently large. Let $V= T^\sigma $ for $\sigma\in [\sigma_-(k), \sigma_+(k)] $, and denote
  \begin{equation*}
    \tau :=
    \begin{cases}
      \frac{1}{2} +\sigma & \text{for } k = 2, \\
      4\cdot \frac{(8k-7)+(16k^2-46k+40)\sigma}
      {16k^2+10k-17} & \text{for } k \ge 3.
    \end{cases}
  \end{equation*}

  Suppose that $\mathcal{T}\subset [T, 2T]$ is a $1$-spaced set with $|\zeta(1/2 + it)| \geq V$ for all $t \in \mathcal{T}$. Write $R := |\mathcal{T}|$ and $G := V^2 (\log T)^{-100} $.

  Partition $\mathbb{R}$ into intervals $(I_\ell)_{\ell\in \mathbb{Z}}$ of length $\frac{2G}{3}$, and let $f_\ell := |\mathcal{T} \cap I_\ell|$. By dyadic pigeonholing, there exists a scale $1 \ll F\ll R$ such that denoting $\mathcal{L}:= \{ \ell \in \mathbb{Z} : f_\ell \sim F\}$, we have $|\mathcal{L}| F \succ R$. Let $\mathcal{S} \subset [T, 2T]$ be the set of leftmost points of the intervals $\overline{I_\ell} \cap [T, 2T]$, for $\ell \in \mathcal{L}$ whose congruence class modulo $4$ is the most abundant one among elements of $\mathcal{L}$. Thus $\mathcal{S}$ is $2G$-separated and $|\mathcal{S}| \succ R/F$. By Lemma \ref{lemma:curx2yevdr}, for every $t \in \mathcal{S}$ we have
  \begin{equation}\label{eq:moment_lower_aux}
    \int_{|u-t| \le G} |\zeta(\tfrac{1}{2}  +iu) |^2 \,d u\gg \frac{FV^2}{\log T} \qquad \text{and} \qquad  \int_{ |u - t | \le G } |\zeta(\tfrac{1}{2} + iu)|^4 \,d u
    \gg \frac{FV^4}{\log T}. 
  \end{equation}
  By Lemma \ref{lem:atkinson} and further dyadic pigeonholing, there exist a scale
  \begin{equation*}
    T^{1/3} \ll N\ll \frac{T}{G^2} (\log T)^2
  \end{equation*}
  and $\mathcal{S}'\subset \mathcal{S}$ still satisfying $|\mathcal{S}'|\succ R/F $ such that for every $t \in \mathcal{S}'$ we have
  \begin{equation} \label{eq:S_t_condition}
    \max_{0\le \Delta\le N} |S_t(\Delta, N)| \succ \frac{FV^2 (TN)^{1/4}}{G}\succ F (TN)^{1/4}.
  \end{equation}

  At this point, take
  \begin{equation*}
    W := N^{\frac{\tau}{1 - 4\sigma}}(\log T)^{-1000} + N.
  \end{equation*}
  We record that $\sigma\in [\sigma_-(k), \sigma_+(k)] $ implies that
  \begin{equation}\label{eq:proof_ranges}
    N\le T^{\frac{3}{5}}, \qquad  N\le W\le N^{\frac{4}{3}}, \qquad  N^{1/2}W^2 \leq T^{3/2}.
  \end{equation}
  In particular, also $W \leq \sqrt{TN}$. Further subdividing $[T, 2T]$, we may partition $\mathcal{S}' $ into sets
  \begin{equation*}
    \mathcal{S}_j\subset \biggl[t_{0, j}, t_{0, j} +  W\sqrt{\frac{T}{N}}\biggr] \qquad \text{with} \qquad t_{0,j } := T  + (j - 1)W\sqrt{\frac{T}{N}} \asymp T
  \end{equation*}
  for $1\le j\le J \asymp \frac{\sqrt{TN}}{W}$. Now, let
  \begin{equation*}
    \mathcal{W}_j := \bigg\{\sqrt{\frac{2N}{\pi}}(\sqrt{t} - \sqrt{t_{0, j}}) : t\in \mathcal{S}_j\bigg\} \subset [0, W].
  \end{equation*}
  The mean value theorem shows that each $\mathcal{W}_j $ is $\gg G\sqrt{N/T} $-spaced. 

  Using \eqref{eq:proof_ranges}, we may apply Corollary \ref{cor:atkinson-model} and Lemma \ref{lem:partial-sums} to \eqref{eq:S_t_condition}, concluding that there exist coefficients $|a_{j, n}| \le 1 $ such that setting 
  \begin{equation*}
    \eta_j :=  -\frac{\pi N}{12 t_{0, j}} \qquad \text{and} \qquad D_{\eta}^j(w) := \sum_{n\sim N } a_{ j, n} e(w \phi_{\eta}(n/N)),
  \end{equation*}
  we have that
  \begin{equation*}
    \sum_{1\le j \le J} \sum_{w\in \mathcal{W}_j } \big|D_{\eta_j}^j(w)\big| \succ F (TN)^{1/4} |\mathcal{S}'| \succ R(TN)^{1/4}.
  \end{equation*}
  Dyadic pigeonholing for the size of $|D_{\eta_j}^j(w)|$ implies that there exists  a scale $Z $
  and subsets $\mathcal{W}_j'\subset \mathcal{W}_j $ such that for every $1\le j \le J$ and $w\in \mathcal{W}_j' $ we have $|D_{\eta_j}^j(w)|\sim Z $ and
  the bounds
  \begin{equation}\label{eq:4}
    Z\succ F(TN)^{1/4}, \quad R_1Z \succ  R(TN)^{1/4}, \qquad \text{ where }  R_1 := \sum_{j } |\mathcal{W}_j'|.
  \end{equation}

  Partition each $\mathcal{W}_j' \subset [0, W]$ into $\prec G^{-1}\sqrt{T/N} $ subsets $\mathcal{W}_{j, \ell}'$ that are $1 $-spaced.
  It can be checked that
  \begin{equation}
    \frac{|\eta_j| N^2}{W}\ll \frac{N^3}{TW}\prec T^{2 - 12\sigma - \tau} = o(1) \qquad \text{as} \qquad T \to \infty.
  \end{equation}
  Thus we may apply Theorem \ref{theorem:curwpy4v0k}, which implies that for each $1\leq j \leq J$ and $\ell$ we have
  \begin{equation}\label{multline:curyjvagxz}
    |\mathcal{W}_{j, \ell}'|\prec_k \frac{N^2}{Z^2} + \frac{N^{\frac{6k - 4}{2k + 1}} W^{\frac{k + 1}{2k + 1}}}{Z^{\frac{8k - 4}{2k + 1}}} \Bigg(  \sup_{1\le D\ll_k 1 + \frac{NW}{T}} \frac{E_2(\mathcal{W}_{j,\ell}'; D)}{\sqrt{D}} \Bigg)^{\frac{1}{2k + 1}} +\frac{N^{k+\frac{1}{2}}W^{\frac{k+1}{2}}\big(1+\frac{NW}{T}\big)^{\frac{1}{4}}}{Z^{2k+1}}.
  \end{equation}

  Let $\mathcal S''\subset\mathcal S'$ consist of the original ordinates corresponding to the elements of the retained sets $\mathcal W'_j$. More precisely, $t \in \mathcal{S}''$ corresponds to $\sqrt{2N/\pi}\big(\sqrt{t} - \sqrt{t_{0, j}}\big) \in \mathcal{W}'_{j}$. Thus $|\mathcal S''|=R_1$ and 
  \begin{equation*}
    \sum_{j, \ell } E_2(\mathcal{W}_{j, \ell}'; D)\ll E_2(\mathcal{S} ''; D\sqrt{T/N}).
  \end{equation*}
  By \eqref{eq:moment_lower_aux} and Corollary \ref{corollary:curxl2f9cb}, this implies that
  \begin{equation}\label{eq:E_2_prelim}
    \sum_{j, \ell } E_2(\mathcal{W}_{j, \ell}';D) \prec  R_1 + \frac{R_1D\sqrt{T/N}}{FV^4} + \frac{T\sqrt{R_1}}{F^{3/2}V^6}.
  \end{equation}
  Separately, the $1 $-spacing of each $\mathcal{W}'_{j, \ell}$ also implies the bound $O(R_1D) $ for the above. 
  We may absorb the middle term of \eqref{eq:E_2_prelim} into the first one because $N\gg T^{1/3} $ and $\tau\le \tfrac{1}{2}+\sigma \leq 6\sigma $, so
  \begin{equation}
    D\sqrt{T/N} \ll \sqrt{T/N} + W\sqrt{N/T} \prec T^{1/3} + T^{\tau - 2\sigma} \prec V^4.
  \end{equation}
  Therefore, we have that
  \begin{equation}
    \sum_{j,\ell } E_2( \mathcal{W}_{j, \ell}'; D) \prec R_1\biggl( 1 + \frac{T}{F^{3/2}V^6\sqrt{R_1}} \biggr).
  \end{equation}
  The total number of subsets $\mathcal{W}_{j,\ell}'$ is
  \begin{equation}
    \prec J \cdot \frac{\sqrt{T/N}}{G} \prec \frac{\sqrt{TN}}{W}\cdot \frac{\sqrt{T/N}}{G}\ll \frac{T}{GW},
  \end{equation}
  so combining with \eqref{multline:curyjvagxz}, we obtain 
  \begin{equation}\label{eq:5}
    R_1\prec_k\frac{T}{GW}\biggl[ \frac{N^2}{Z^2} +\frac{N^{3-\frac{2}{k}}W^{\frac{k+1}{2k}}}{Z^{4-\frac{2}{k}}} \biggl(1+\frac{T}{F^{3/2}V^6\sqrt{R_1}}\biggr)^{\frac{1}{4k}} +\frac{N^{k+\frac{1}{2}}W^{\frac{k+1}{2}}\big(1+\frac{NW}{T}\big)^{1/4}}
    {Z^{2k+1}}
    \biggr].
  \end{equation}
  Rearranging to clear the power of $R_1 $, this gives
  \begin{align}\label{eq:6}
    R_1\prec_k
    \frac{T}{GW}\biggl[ \frac{N^2}{Z^2} +\frac{N^{3-\frac{2}{k}}W^{\frac{k + 1}{2k}}}{Z^{4-2/k}} + \frac{N^{k+\frac{1}{2}}W^{\frac{k + 1}{2}}\big(1+\frac{NW}{T}\big)^{\frac{1}{4}}} {Z^{2k+1}} \biggr]+\biggl( \frac{T^{8k+2}N^{24k-16}}{G^{8k}W^{4k-4}Z^{32k-16}F^3V^{12}} \biggr)^{\frac{1}{8k + 1}}.
  \end{align}

  Multiplying by $Z/(TN)^{1/4} $ and using that $R\prec ZR_1/(TN)^{1/4} $ and $Z\succ F(TN)^{1/4} $, we get 
  \begin{equation}
    R\prec_k \frac{T^{\frac{1}{2}}N^{\frac{3}{2}}}{GW} + \frac{T^{\frac{1}{2k}}N^{2 - \frac{3}{2k}}}{GW^{\frac{k - 1}{2k}}}
    + \frac{T^{\frac{3}{4} - \frac{k}{2}}N^{\frac{k}{2} + \frac{1}{4}}W^{\frac{k - 1}{2}} \big(1+\frac{NW}{T}\big)^{\frac{1}{4}}}{G}
    + \biggl( \frac{T^6N^{16k - 12}}{G^{8k}W^{4k - 4}V^{12}}  \biggr)^{\frac{1}{8k + 1}}.
  \end{equation}
  Using that $N\prec T^{1 - 4\sigma} $ and $ N^{\frac{\tau}{1 - 4\sigma}} \prec  W \prec N^{\frac{\tau}{1 - 4\sigma}}$, this leads to
  \begin{equation}\label{eq:curyjytduj}
    R\prec_k T^{2 - 8\sigma - \tau} + T^{2 - \frac{1}{k} - (10 - \frac{6}{k})\sigma - \frac{k- 1}{2k}\tau}
    + T^{1 - (2k + 4)\sigma + \frac{2k - 1}{4}\tau}
    + T^{\frac{16k - 6 - (80k - 36)\sigma - 4(k - 1)\tau}{8k + 1}}.
  \end{equation}
  Our choice of \(\tau\) satisfies 
  \begin{equation}\label{eq:7}
    \frac{16k - 6 - (80k - 36)\sigma - 4(k - 1)\tau}{8k + 1} = E_k(\sigma).
  \end{equation}
  For \(k\ge3\), the third exponent in \eqref{eq:curyjytduj} also equals \(E_k(\sigma)\), while the
  second is at most \(E_k(\sigma)\) precisely when \(\sigma\le\sigma_+(k)\); the first is smaller throughout
  the stated range. For $k=2$, the second and third exponents both equal $\frac{11}{8} -\frac{29\sigma}{4}$,
  which is $\le E_2(\sigma) $ for $\sigma\le \frac{5}{38} $. The first exponent in \eqref{eq:curyjytduj} is also smaller than the third for $\sigma \in \big[\frac{7}{60}, \frac{5}{38}\big]$. Hence we obtain the desired bound $R\prec_k T^{E_k(\sigma)}$.
\end{proof}

\section{Improved moment bounds}

An immediate consequence of our large value estimates, combined with work of Trudgian and Yang \cite{MR4864482},
is the following improvement of the twelfth moment bound of Heath-Brown \cite{MR517737}.
\begin{theorem}[Twelfth moment bound]\label{theorem:curykapvmu}
  For every $T \geq 1$ and $\varepsilon > 0$, we have
  \begin{equation}
    \int_{0}^T |\zeta(\tfrac{1}{2} + it)|^{12} \,d t \ll_\varepsilon T^{2 - \frac{520}{445121} + \varepsilon}.
  \end{equation}
\end{theorem}
For reference, $\frac{520}{445121} $ is slightly greater than $\frac{1}{1000}$.
\begin{proof}
  Theorem \ref{theorem:curxme9sty} with $k = 6 $ states that
  \begin{equation}\label{eq:8}
    N(T,T^\sigma)\prec
    T^{\frac{1070-6164\sigma}{619}}
    \qquad \text{for} \qquad \tfrac{451}{3688}\le\sigma\le\tfrac{21}{158}.
  \end{equation}
  Combined with \cite[Lemma 5.1]{MR4864482}, using the exponent pair $\big(\frac{64}{135}, \frac{122}{243}\big) $
   corresponding to $m = 10 $ in their given sequence, we also have
  \begin{equation}
    N(T,V)\prec TV^{-6}+T^{\frac{593}{288}}V^{-\frac{3587}{288}}.
  \end{equation}
  Applying the subconvexity bound of Bourgain \cite{B}, it can be checked that for $\sigma_1 = \frac{451}{3688} $ and $\sigma_2 = \frac{58907}{445121} $ we have
  \begin{align}
    \int_{T}^{2T} |\zeta(\tfrac{1}{2} + it)|^{12} \,d t &\prec \sup_{\sigma\in \left[0, \frac{13}{84}\right]} T^{12\sigma}N(T,T^\sigma) \\
    &\hspace{-3cm}\prec \max\Big(\sup_{\sigma\in [0, \sigma_1]} T^{1 + 8\sigma}, \sup_{\sigma\in[\sigma_1, \sigma_2]} T^{\frac{1070 + 1264\sigma}{619}}, \sup_{\sigma\in \left[\sigma_2, \frac{13}{84}\right]} T^{\max\left(1 + 6\sigma, \frac{593 - 131\sigma}{288}\right)}\Big)  
    \prec T^{2 - \frac{520}{445121} + \varepsilon}.
  \end{align}
  Splitting into dyadic intervals yields Theorem \ref{theorem:curykapvmu}.
\end{proof}

The claimed improvements for all intermediate moments is now simple to deduce.

\begin{proof}[Proof of Theorem \ref{cor:gen_moments}]
    The existence of some $\delta_k > 0$ for each $4<k\leq 12$ follows from interpolation between the fourth moment and the new twelfth moment bound in Theorem \ref{theorem:curykapvmu}. 
\end{proof}

\begin{proof}[Proof of Theorem \ref{cor:sixth_moment}]
    Let $V_0=T^{7/60}$ and $V_1=T^{5/38}$.
    Dyadically decomposing the level sets $|\zeta(1/2+it)| \geq V$, using Theorem \ref{theorem:curvbfoast} for
     $V \in [V_0,V_1]$, and \eqref{eq:fourth_twelfth_moment_bound} for $V$ outside of this range, we get
    \begin{equation*}
      \int_T^{2T} |\zeta(\tfrac12+it)|^6\,dt
      \ll_\varepsilon T^\varepsilon
      \bigl(TV_0^2+T^{\frac{24}{17}}V_0^{-\frac{26}{17}}+T^2V_1^{-6}\bigr)
      \ll_\varepsilon T^{\frac{37}{30}+\varepsilon}.
    \end{equation*}
    Summing over dyadic intervals gives the desired bound, since $\frac{37}{30}=\frac{5}{4}-\frac{1}{60}$.
\end{proof}

\bibliography{refs_small}{} \bibliographystyle{alpha}

@article{Jutila1991MeanValues,
  author = {Jutila, M.},
  title = {Mean value estimates for exponential sums with applications to {$L$}-functions},
  journal = {Acta Arith.},
  fjournal = {Acta Arithmetica},
  volume = {57},
  number = {2},
  year = {1991},
  pages = {93--114},
  doi = {10.4064/aa-57-2-93-114},
  url = {https://doi.org/10.4064/aa-57-2-93-114},
  display-string = {(1991) Jutila, Matti - Mean value estimates for exponential sums with applications to L-functions}
}

@article{DasguptaLeungYoung2024,
  author = {Dasgupta, A. and Leung, W. H. and Young, M. P.},
  title = {The second moment of the {$GL_3$} standard {$L$}-function on the critical line},
  journal = {Preprint},
  year = {2024},
  note = {\href{https://arxiv.org/abs/2407.06962}{\texttt{arXiv:2407.06962}}. To appear in Compositio Mathematica},
  display-string = {(2024) Dasgupta, Agniva and Leung, Wing Hong and Young, Matthew P. - The second moment of the GL3 standard L-function on the critical line}
}

@article {AggarwalLeungMunshi2026,
    AUTHOR = {Aggarwal, K. and Leung, W. H. and Munshi, R.},
     TITLE = {Short second moment bound and subconvexity for {$\mathrm{GL}(3)$} {$L$}-functions},
   JOURNAL = {J. Eur. Math. Soc. (JEMS)},
  FJOURNAL = {Journal of the European Mathematical Society (JEMS)},
    VOLUME = {28},
      YEAR = {2026},
    NUMBER = {11},
     PAGES = {4797--4871},
      ISSN = {1435-9855,1435-9863},
   MRCLASS = {11F66 (11M41)},
  MRNUMBER = {5126271},
       DOI = {10.4171/jems/1667},
       URL = {https://doi.org/10.4171/jems/1667},
}

@article{Iwaniec2024PartialSixthMoment,
  author = {Iwaniec, H.},
  title = {The sixth power moment of partial {Riemann} zeta function and {Kloosterman} sums},
  journal = {Funct. Approx. Comment. Math.},
  fjournal = {Functiones et Approximatio Commentarii Mathematici},
  volume = {71},
  number = {1},
  year = {2024},
  pages = {121--125},
  doi = {10.7169/facm/2114},
  url = {https://doi.org/10.7169/facm/2114},
  display-string = {(2024) Iwaniec, Henryk - The sixth power moment of partial Riemann zeta function and Kloosterman sums}
}

@article{Guth2025LargeValues,
  author = {Guth, L.},
  title = {Large value estimates in number theory, harmonic analysis, and computer science},
  journal = {Preprint},
  year = {2025},
  note = {\href{https://arxiv.org/abs/2503.07410}{\texttt{arXiv:2503.07410}}},
  display-string = {(2025) Guth, Larry - Large value estimates in number theory, harmonic analysis, and computer science}
}

@article{FuGuthMaldague2023,
  author = {Fu, Y. and Guth, L. and Maldague, D.},
  title = {Decoupling inequalities for short generalized {Dirichlet} sequences},
  journal = {Anal. PDE},
  fjournal = {Analysis \& PDE},
  volume = {16},
  number = {10},
  year = {2023},
  pages = {2401--2464},
  doi = {10.2140/apde.2023.16.2401},
  url = {https://doi.org/10.2140/apde.2023.16.2401},
  eprint = {2104.00856},
  archiveprefix = {arXiv},
  primaryclass = {math.CA},
  display-string = {(2023) Fu, Yuqiu and Guth, Larry and Maldague, Dominique - Decoupling inequalities for short generalized Dirichlet sequences}
}

@article {MR5037841,
    AUTHOR = {Guth, L. and Maynard, J.},
     TITLE = {New large value estimates for {D}irichlet polynomials},
   JOURNAL = {Ann. of Math. (2)},
  FJOURNAL = {Annals of Mathematics. Second Series},
    VOLUME = {203},
      YEAR = {2026},
    NUMBER = {2},
     PAGES = {623--675},
      ISSN = {0003-486X,1939-8980},
   MRCLASS = {11M26 (11N05)},
  MRNUMBER = {5037841},
MRREVIEWER = {Timothy\ S.\ Trudgian},
       DOI = {10.4007/annals.2026.203.2.6},
       URL = {https://doi.org/10.4007/annals.2026.203.2.6},
  display-string = {(2026) Guth, Larry and Maynard, James - New large value estimates for Dirichlet polynomials},
}

@misc{zbMATH03683696,
 author = {Iwaniec, H.},
 title = {Fourier coefficients of cusp forms and the {Riemann} zeta-function},
 year = {1980},
 language = {English},
 howpublished = {S{\'e}min. {Th{\'e}or}. {Nombres} {Bordeaux} 1979-1980, {Exp}. {No}. 18, 36 pp.},
 url = {https://eudml.org/doc/182071},
 zbMATH = {3683696},
 Zbl = {0438.10030},
  display-string = {(1980) Iwaniec, Henryk - Fourier coefficients of cusp forms and the Riemann zeta-function}
}

@article {MR517737,
    AUTHOR = {Heath-Brown, D. R.},
     TITLE = {The twelfth power moment of the {R}iemann-function},
   JOURNAL = {Quart. J. Math. Oxford Ser. (2)},
  FJOURNAL = {The Quarterly Journal of Mathematics. Oxford. Second Series},
    VOLUME = {29},
      YEAR = {1978},
    NUMBER = {116},
     PAGES = {443--462},
      ISSN = {0033-5606,1464-3847},
   MRCLASS = {10H05},
  MRNUMBER = {517737},
MRREVIEWER = {M.\ I.\ Israilov},
       DOI = {10.1093/qmath/29.4.443},
       URL = {https://doi.org/10.1093/qmath/29.4.443},
  display-string = {(1978) Heath-Brown, D. R. - The twelfth power moment of the Riemann-function},
}

@article{B,
 author = {Bourgain, J.},
 display-string = {(2017) Bourgain, J. - Decoupling, exponential sums and the Riemann zeta function},
 doi = {10.1090/jams/860},
 fjournal = {Journal of the American Mathematical Society},
 issn = {0894-0347,1088-6834},
 journal = {J. Amer. Math. Soc.},
 mrclass = {11M06 (11L07)},
 mrnumber = {3556291},
 mrreviewer = {Kostadinka\ Lapkova},
 number = {1},
 pages = {205--224},
 title = {Decoupling, exponential sums and the {R}iemann zeta function},
 url = {https://doi.org/10.1090/jams/860},
 volume = {30},
 year = {2017}
}

@book{Stein1993,
 author = {Stein, E. M.},
 title = {Harmonic Analysis: Real-Variable Methods, Orthogonality, and Oscillatory Integrals},
 series = {Princeton Mathematical Series},
 volume = {43},
 publisher = {Princeton University Press},
 address = {Princeton, NJ},
 year = {1993},
 isbn = {0-691-03216-5},
  display-string = {(1993) Stein, Elias M. - Harmonic Analysis: Real-Variable Methods, Orthogonality, and Oscillatory Integrals},
}

@article {MR3994723,
    AUTHOR = {Kiral, E. M. and Petrow, I. and Young, M. P.},
     TITLE = {Oscillatory integrals with uniformity in parameters},
   JOURNAL = {J. Th\'eor. Nombres Bordeaux},
  FJOURNAL = {Journal de Th\'eorie des Nombres de Bordeaux},
    VOLUME = {31},
      YEAR = {2019},
    NUMBER = {1},
     PAGES = {145--159},
      ISSN = {1246-7405,2118-8572},
   MRCLASS = {41A60 (42A38)},
  MRNUMBER = {3994723},
MRREVIEWER = {Abdallah\ Benaissa},
       DOI = {10.5802/jtnb.1072},
       URL = {https://doi.org/10.5802/jtnb.1072},
  display-string = {(2019) Kiral, E. M. and Petrow, I. and Young, M. P. - Oscillatory integrals with uniformity in parameters},
}

@article{2024arXiv2401.13981,
    title={Squarefree numbers in short intervals}, 
    author={Pandey, M.},
    journal = {Preprint},
    year={2024},
    note = {\href{http://arxiv.org/abs/2401.13981}{\texttt{arXiv:2401.13981}}},
  display-string = {(2024) Pandey, M. - Squarefree numbers in short intervals}
}

@article {MR1575368,
    AUTHOR = {Hardy, G. H. and Littlewood, J. E.},
     TITLE = {The {A}pproximate {F}unctional {E}quation in the {T}heory of
              the {Z}eta-{F}unction, with {A}pplications to the
              {D}ivisor-{P}roblems of {D}irichlet and {P}iltz},
   JOURNAL = {Proc. London Math. Soc. (2)},
  FJOURNAL = {Proceedings of the London Mathematical Society. Second Series},
    VOLUME = {21},
      YEAR = {1923},
     PAGES = {39--74},
      ISSN = {0024-6115},
   MRCLASS = {99-04},
  MRNUMBER = {1575368},
       DOI = {10.1112/plms/s2-21.1.39},
       URL = {https://doi.org/10.1112/plms/s2-21.1.39},
  display-string = {(1923) Hardy, G. H. and Littlewood, J. E. - The Approximate Functional Equation in the Theory of
              the Zeta-Function, with Applications to the
              Divisor-Problems of Dirichlet and Piltz},
}

@article {MR337775,
    AUTHOR = {Montgomery, H. L. and Vaughan, R. C.},
     TITLE = {Hilbert's inequality},
   JOURNAL = {J. London Math. Soc. (2)},
  FJOURNAL = {Journal of the London Mathematical Society. Second Series},
    VOLUME = {8},
      YEAR = {1974},
     PAGES = {73--82},
      ISSN = {0024-6107,1469-7750},
   MRCLASS = {10C15 (10E25 10H05)},
  MRNUMBER = {337775},
MRREVIEWER = {E.\ Grosswald},
       DOI = {10.1112/jlms/s2-8.1.73},
       URL = {https://doi.org/10.1112/jlms/s2-8.1.73},
  display-string = {(1974) Montgomery, H. L. and Vaughan, R. C. - Hilbert's inequality},
}

@book {MontgomeryVaughanII,
    AUTHOR = {Montgomery, H. L. and Vaughan, R. C.},
     TITLE = {Multiplicative number theory. {II}. {P}rimes and sieves},
    SERIES = {Cambridge Studies in Advanced Mathematics},
    VOLUME = {218},
 PUBLISHER = {Cambridge University Press, Cambridge},
      YEAR = {2026},
     PAGES = {xiv+461},
  display-string = {(2026) Montgomery, Hugh L. and Vaughan, Robert C. - Multiplicative number theory. II. Primes and sieves}
}

@book {MR1230387,
    AUTHOR = {Ivi\'c, A.},
     TITLE = {Lectures on mean values of the {R}iemann zeta function},
    SERIES = {Tata Institute of Fundamental Research Lectures on Mathematics
              and Physics},
    VOLUME = {82},
 PUBLISHER = {Tata Institute of Fundamental Research, Bombay; by
              Springer-Verlag, Berlin},
      YEAR = {1991},
     PAGES = {viii+363},
      ISBN = {3-540-54748-7},
   MRCLASS = {11M06 (11-02 11F72)},
  MRNUMBER = {1230387},
MRREVIEWER = {D.\ R.\ Heath-Brown},
  display-string = {(1991) Ivic, A. - Lectures on mean values of the Riemann zeta function},
}

@book {MR1994094,
    AUTHOR = {Ivi\'c, A.},
     TITLE = {The {R}iemann zeta-function. {Theory} and applications},
     NOTE = {{R}eprint of the 1985 original},
 PUBLISHER = {Dover Publications, Inc., Mineola, NY},
      YEAR = {2003},
     PAGES = {xxii+517},
      ISBN = {0-486-42813-3},
   MRCLASS = {11M06 (11M26)},
  MRNUMBER = {1994094},
  display-string = {(2003) Ivic, Aleksandar - The Riemann zeta-function},
}

@article {MR31963,
    AUTHOR = {Atkinson, F. V.},
     TITLE = {The mean-value of the {R}iemann zeta function},
   JOURNAL = {Acta Math.},
  FJOURNAL = {Acta Mathematica},
    VOLUME = {81},
      YEAR = {1949},
     PAGES = {353--376},
      ISSN = {0001-5962,1871-2509},
   MRCLASS = {10.0X},
  MRNUMBER = {31963},
MRREVIEWER = {R.\ Bellman},
       DOI = {10.1007/BF02395027},
       URL = {https://doi.org/10.1007/BF02395027},
  display-string = {(1949) Atkinson, F. V. - The mean-value of the Riemann zeta function},
}

@article {MR1575391,
    AUTHOR = {Ingham, A. E.},
     TITLE = {Mean-{V}alue {T}heorems in the {T}heory of the {R}iemann
              {Z}eta-{F}unction},
   JOURNAL = {Proc. London Math. Soc. (2)},
  FJOURNAL = {Proceedings of the London Mathematical Society. Second Series},
    VOLUME = {27},
      YEAR = {1927},
    NUMBER = {4},
     PAGES = {273--300},
      ISSN = {0024-6115},
   MRCLASS = {99-04},
  MRNUMBER = {1575391},
       DOI = {10.1112/plms/s2-27.1.273},
       URL = {https://doi.org/10.1112/plms/s2-27.1.273},
  display-string = {(1927) Ingham, A. E. - Mean-Value Theorems in the Theory of the Riemann
              Zeta-Function},
}

@article {MR4864482,
    AUTHOR = {Trudgian, T. S. and Yang, A.},
     TITLE = {Toward optimal exponent pairs},
   JOURNAL = {Math. Comp.},
  FJOURNAL = {Mathematics of Computation},
    VOLUME = {94},
      YEAR = {2025},
    NUMBER = {353},
     PAGES = {1467--1502},
      ISSN = {0025-5718,1088-6842},
   MRCLASS = {11L07 (11M06 11T23)},
  MRNUMBER = {4864482},
MRREVIEWER = {Gregory\ Debruyne},
       DOI = {10.1090/mcom/3993},
       URL = {https://doi.org/10.1090/mcom/3993},
  display-string = {(2025) Trudgian, Timothy S. and Yang, Andrew - Toward optimal exponent pairs},
}

@article {MR268130,
    AUTHOR = {Montgomery, H. L.},
     TITLE = {Mean and large values of {D}irichlet polynomials},
   JOURNAL = {Invent. Math.},
  FJOURNAL = {Inventiones Mathematicae},
    VOLUME = {8},
      YEAR = {1969},
     PAGES = {334--345},
      ISSN = {0020-9910,1432-1297},
   MRCLASS = {10.40},
  MRNUMBER = {268130},
MRREVIEWER = {B.\ Berlowitz},
       DOI = {10.1007/BF01404637},
       URL = {https://doi.org/10.1007/BF01404637},
}

@article {MR230694,
    AUTHOR = {Hal\'asz, G.},
     TITLE = {\"Uber die {M}ittelwerte multiplikativer zahlentheoretischer
              {F}unktionen},
   JOURNAL = {Acta Math. Acad. Sci. Hungar.},
  FJOURNAL = {Acta Mathematica. Academiae Scientiarum Hungaricae},
    VOLUME = {19},
      YEAR = {1968},
     PAGES = {365--403},
      ISSN = {0001-5954,1588-2632},
   MRCLASS = {10.43},
  MRNUMBER = {230694},
MRREVIEWER = {H.-J.\ Kanold},
       DOI = {10.1007/BF01894515},
       URL = {https://doi.org/10.1007/BF01894515},
}

@article {MR292774,
    AUTHOR = {Huxley, M. N.},
     TITLE = {On the difference between consecutive primes},
   JOURNAL = {Invent. Math.},
  FJOURNAL = {Inventiones Mathematicae},
    VOLUME = {15},
      YEAR = {1972},
     PAGES = {164--170},
      ISSN = {0020-9910,1432-1297},
   MRCLASS = {10H15},
  MRNUMBER = {292774},
MRREVIEWER = {H.\ L.\ Montgomery},
       DOI = {10.1007/BF01418933},
       URL = {https://doi.org/10.1007/BF01418933},
}
\end{document}